\documentclass[12pt, reqno]{amsart}

\usepackage{amssymb, amsmath, amsthm}
\usepackage{hyperref, enumitem}
\usepackage{thmtools}
\usepackage[capitalize, nameinlink]{cleveref}
\usepackage[space]{cite}
\usepackage[centering, margin=1.1in]{geometry}
\usepackage{xcolor, graphicx}
\usepackage{bbm}
\usepackage{mathrsfs}
\usepackage{pgf}
\usepackage{caption}
\usepackage{subcaption}
\usepackage{comment}
\usepackage[utf8]{inputenc}\DeclareUnicodeCharacter{2212}{-}

\usepackage{mathrsfs}
\usepackage{amssymb}
\usepackage{dsfont}
\usepackage{verbatim}
\usepackage{url}
\usepackage{mathtools}
\usepackage{etoolbox}
\usepackage{leftidx}

\def\N{\mathbb {N}}
\def\Z{\mathbb {Z}}
\def\Q{\mathbb {Q}}
\def\R{\mathbb {R}}
\def\C{\mathbb {C}}

\def\isom{\simeq}

\def\bs{\backslash}

\DeclareMathOperator{\Ad}{Ad}

\DeclareMathOperator{\diag}{diag}

\DeclareMathOperator{\Gal}{Gal}

\DeclareMathOperator{\Isom}{Isom}

\def\sgn{\qopname\relax o{sgn}}

\DeclareMathOperator{\supp}{supp}

\DeclareMathOperator{\vol}{vol}

\newcommand\abs[1]{\left| {#1} \right|}
\newcommand\norm[1]{\left\Vert {#1} \right\Vert}

\newcommand{\xdashrightarrow}[2][]{\ext@arrow 0359\rightarrowfill@@{#1}{#2}}

\def\cprime{\ensuremath{'}}

\DeclareRobustCommand
  \rddots{\mathinner{\mkern1mu\raise\p@
    \vbox{\kern7\p@\hbox{.}}\mkern2mu
    \raise4\p@\hbox{.}\mkern2mu\raise7\p@\hbox{.}\mkern1mu}}

\newcommand\PGSV{\mathrm{PGSV}}
\newcommand\GSV{\mathrm{GSV}}
\newcommand\SV{\mathrm{SV}}
\newcommand\SL{\mathrm{SL}}
\newcommand\GL{\mathrm{GL}}

\newcommand\Sp{\mathrm{Sp}}
\newcommand\GSp{\mathrm{GSp}}
\newcommand\PGSp{\mathrm{PGSp}}

\newcommand\Spin{\mathrm{Spin}}

\newcommand\SO{\mathrm{SO}}

\newcommand\GSpin{\mathrm{GSpin}}

\newcommand\CC{\mathbb{C}}

\newcommand\HH{\mathbb{H}}

\newcommand\QQ{\mathbb{Q}}
\newcommand\RR{\mathbb{R}}

\newcommand\ZZ{\mathbb{Z}}

\newcommand\bH{\mathbf{H}}

\DeclareMathAlphabet{\mathcal}{OMS}{cmsy}{m}{n}

\newcommand\cF{\mathcal{F}}

\newcommand\cH{\mathcal{H}}

\newcommand\cL{\mathcal{L}}

\newcommand\cN{\mathcal{N}}

\newcommand\cP{\mathcal{P}}

\newcommand\cT{\mathcal{T}}

\newcommand\sL{\mathscr{L}}

\newcommand\diff{\mathop{}\!\mathrm{d}}

\newcommand\dcross{\diff^{\kern-2pt\raisebox{1pt}{$\times$}}\kern-8pt}

\newcommand\uG{\underline{G}}
\newcommand\uH{\underline{H}}

\newcommand\blfootnote{\xdef\@thefnmark{}\@footnotetext}

\DeclareFontFamily{U}{wncy}{}
\DeclareFontShape{U}{wncy}{m}{n}{<->wncyr10}{}
\DeclareSymbolFont{mcy}{U}{wncy}{m}{n}
\DeclareMathSymbol{\Sha}{\mathord}{mcy}{"58}

\renewcommand{\H}{\mathbb{H}}

\newcommand{\uV}{\underline{V}}

\newcommand{\1}{\mathbbm{1}}

\newcommand{\bi}{\mathbf{i}}
\newcommand{\bj}{\mathbf{j}}
\newcommand{\bk}{\mathbf{k}}

\renewcommand{\epsilon}{\varepsilon}

\DeclareFontFamily{U}{mathx}{\hyphenchar\font45}
\DeclareFontShape{U}{mathx}{m}{n}{
      <5> <6> <7> <8> <9> <10>
      <10.95> <12> <14.4> <17.28> <20.74> <24.88>
      mathx10
      }{}
\DeclareSymbolFont{mathx}{U}{mathx}{m}{n}
\DeclareFontSubstitution{U}{mathx}{m}{n}
\DeclareMathAccent{\widecheck}{0}{mathx}{"71}
\DeclareMathAccent{\wideparen}{0}{mathx}{"75}

\newcommand{\autoalign}[1]{
  \begin{minipage}{\linewidth}
    \setlength{\parindent}{0pt}
    \setlength{\parfillskip}{0pt} 
    \linespread{1.2}\selectfont % Add vertical spacing
    \sloppy
    $\displaystyle #1$
  \end{minipage}
}

\DeclareMathOperator{\Pit}{Pit}
\DeclareMathOperator{\Iso}{Iso}

\declaretheorem[name=Theorem]{theorem}
\declaretheorem[name=Question]{question}
\declaretheorem[name=Lemma]{lemma}
\declaretheorem[name=Proposition]{proposition}
\declaretheorem[name=Corollary]{corollary}
\declaretheorem[name=Definition]{definition}

\declaretheorem[name=Remark]{remark}

\declaretheorem[name=Claim, numbered=no]{claim*}

\newlist{theoremlist}{enumerate}{1}
\setlist[theoremlist]{label=(\alph{theoremlisti}),
                  ref=\thetheorem \ (\alph{theoremlisti}),
                  noitemsep}
                  
\newlist{lemmalist}{enumerate}{1}
\setlist[lemmalist]{label=(\alph{lemmalisti}),
                  ref=\thelemma \ (\alph{lemmalisti}),
                  noitemsep}
                  
\newlist{propositionlist}{enumerate}{1}
\setlist[propositionlist]{label=(\alph{propositionlisti}),
                  ref=\theproposition \ (\alph{propositionlisti}),
                  noitemsep}

\newlist{corollarylist}{enumerate}{1}
\setlist[corollarylist]{label=(\alph{corollarylisti}),
                  ref=\thecorollary \ (\alph{corollarylisti}),
                  noitemsep}
                                                      
\Crefname{listthm}{Theorem}{Theorems}
\Crefname{listlem}{Lemma}{Lemmas}
\Crefname{listprop}{Proposition}{Propositions}
\Crefname{listcor}{Corollary}{Corollaries}

\addtotheorempostheadhook[theorem]{\crefalias{theoremlisti}{listthm}}
\addtotheorempostheadhook[lemma]{\crefalias{lemmalisti}{listlem}}
\addtotheorempostheadhook[proposition]{\crefalias{propositionlisti}{listprop}}
\addtotheorempostheadhook[corollary]{\crefalias{corollarylisti}{listcor}}

\crefname{section}{Section}{Sections}
\Crefname{section}{Section}{Sections}

\crefname{appendix}{Appendix}{Appendices}
\Crefname{appendix}{Appendix}{Appendices}

\begin{document}

\title{Mass equidistribution for lifts on hyperbolic $4$-manifolds}

\author{Alexandre de Faveri}
\address{EPFL SB MATH, Station 10, 1015 Lausanne, Switzerland}
\email{\url{alexandre.defaveri@epfl.ch}}
\thanks{A.F.\ and Z.S.\ acknowledge support by the NSF grant DMS-1926686.}

\author{Zvi Shem-Tov}
\address{Department of Mathematics, Ben-Gurion University
of the Negev, Be’er Sheva, Israel}
\email{\url{zvishem@bgu.ac.il}}
\thanks{Z.S.\ acknowledges support by the ERC grant HomDyn (grant no.\ 833423).}

\begin{abstract}
    This paper is part of a broader project to establish the quantum unique ergodicity conjecture for Hecke--Maass forms on arithmetic hyperbolic $4$-manifolds. The conjecture is known in dimensions $2$ and $3$, but dimensions $4$ and higher present a new difficulty: ruling out concentration along certain homogeneous submanifolds whose stabilizers are non-tempered.
    
    We overcome it here for the Pitale lifts, a family of non-holomorphic analogues of Saito--Kurokawa lifts. Our main new ingredient is a construction of an amplifier with exceptional geometric properties. To the best of our knowledge, this is the first successful use of amplification to rule out concentration along a non-tempered subgroup. 
\end{abstract}

\maketitle

%%%%%%%%%%%%%%%%%%%%%%%%%%%%%%%%%%%%%%%%%%%%%%%%%%%%%%%

\section{Introduction}

\subsection{Background}

A fundamental result in quantum chaos is the celebrated quantum ergodicity (QE) theorem of {\v S}nirel{\cprime}man \cite{Schnirelman:Avg_QUE}, Zelditch \cite{Zelditch:SL2_Lift_QE}, and Colin de Verdi{\`e}re \cite{CdV:Avg_QUE}. Let $M$ be a compact Riemannian manifold with ergodic geodesic flow, and let $\phi_j$ be any orthonormal basis of $L^2(M)$ consisting of Laplace eigenfunctions. For each $j$ define the probability measure
\begin{equation*}
    d\mu_j(x)=\abs{\phi_j(x)}^2 d\vol_M(x),
\end{equation*}
where $\vol_M$ is the standard probability measure on $M$. The QE theorem states that there exists a subsequence $j_k$ of density $1$ such that $\mu_{j_k}$ converges weak-$*$ to $\vol_M$ as $k\to\infty$. 

While in this generality the full sequence $\mu_j$ might not converge \cite{Hassell:NonUniqueBilliard}, Rudnick and Sarnak conjectured \cite{RudnickSarnak:Conj_QUE} that if $M$ has negative sectional curvature, then $\mu_j$ converges weak-$*$ to $d\vol_M$. This is called the quantum unique ergodicity (QUE) conjecture.   

In the seminal work \cite{Lindenstrauss:SL2_QUE}, Lindenstrauss proved QUE for Hecke--Maass forms\footnote{For progress on the non-arithmetic case, see \cite{KM25, DJ18, AN07}.} on congruence hyperbolic surfaces using his striking measure rigidity theorem for diagonal actions on homogeneous spaces. In fact, his result applies also for non-compact congruence surfaces such as $\SL_2(\Z)\bs\HH^2$, and QUE for those surfaces follows after combining with subsequent work of Soundararajan \cite{Sound:EscapeOfMass} showing that any weak-$*$ limit of the $\mu_j$ is a probability measure (``non-escape of mass''). 

Measure rigidity tools due to Einsiedler and Lindenstrauss \cite{EinsiedlerLindenstrauss:RankOne} also played a central role in \cite{ShemTovSilberman:HH_AQUE_preprint}, where the QUE conjecture was established for sequences of Hecke--Maass forms on quotients of $\HH^3$. However, unlike in the case of surfaces, the measure classification results allow for measures giving positive mass to proper submanifolds, a phenomenon called ``scarring''. The main challenge in \cite{ShemTovSilberman:HH_AQUE_preprint} was eliminating this possibility. 

For Hecke--Maass forms in $\H^4$, Shem-Tov and Silberman \cite{ShemTovSilberman:AQUE4d} combined the methods of \cite{ShemTovSilberman:HH_AQUE_preprint} with results of Marshall \cite{marshallupperboundsmaassforms} to establish that the only possible obstruction to QUE is scarring along totally geodesic submanifolds of dimension $3$. For non-uniform quotients this also requires non-escape of mass, which was recently proved by the authors \cite{deFaveriShemTov:NonEscape4d}. This raises the following question.

\begin{question}\label{question1}
    Let $M$ be a congruence quotient $M=\Gamma\bs\HH^4$, and let $\mu$ be a weak-$*$ limit of the probability measures $|\phi_j(z)|^2 \, d\vol_M(z)$, where the $\phi_j$ are Hecke--Maass forms on $M$. Can $\mu$ give positive measure to a proper submanifold of $M$?
\end{question}

The existing general techniques for establishing non-concentration of Hecke--Maass forms, based on the \emph{amplification method}, fall short of addressing this problem. The reason can be understood in light of the work \cite{marshallupperboundsmaassforms} of Marshall, which gives a systematic way to obtain non-concentration on submanifolds corresponding to orbits of \emph{weakly-small} subgroups (see \cref{smallness-def}) of the isometry group. The challenge in answering \cref{question1} comes from the fact that, unlike in dimensions $2$ and $3$, the isometry group of $\HH^4$ contains subgroups which are not weakly-small (e.g.\ $\SO(1,3)\subset\SO(1,4)$).

The notion of weak-smallness is a slight strengthening of that of \emph{temperedness} (see \cite{BK15, marshallupperboundsmaassforms} for further discussion of this notion). A reductive subgroup $\uH\subset \uG$ is called \emph{tempered} if it satisfies an inequality of the form  
\begin{equation}\label{eq:weak-smallness-intro}
    \norm{\eta}_{\uH}^*\le\frac{1}{2}\norm{\eta}_{\uG}^*
\end{equation}
for all cocharacters $\eta$ of a maximal split torus in $\uH$ (see \cref{smallness-def} for details). 
Let $G_p$ and $H_p$ denote respectively the groups of $\QQ_p$-points of $\uG$ and $\uH$, and $K_p$ be a compatible maximal compact subgroup of $G_p$. Then $p^{2\norm{\eta}_{\uH}^*}$ is roughly the size of the intersection of $K_p\eta(p)K_p$ with $H_p$ in $G_p/K_p$. 
The constant $\frac{1}{2}$ appears in \eqref{eq:weak-smallness-intro} due to the fact that the Hecke operator $\tau = \mathbf{1}_{K_p \eta(p) K_p}$ (viewed as a convolution kernel) acts on a typical Hecke eigenfunction with eigenvalue $\ll \norm{\tau}_2 \asymp p^{\norm{\eta}^*_{\uG}}$. However, for some subgroups $\uH \subset \uG$ that arise in our context (those locally isomorphic to $\SO(1,3)$), we only have the weaker bound 
\begin{equation}\label{eq:weak-smallness-intro2}
    \norm{\eta}_{\uH}^*\le\frac{2}{3}\norm{\eta}_{\uG}^*.
\end{equation}
For Hecke eigenvalues of generic size, the weaker bound \eqref{eq:weak-smallness-intro2} is insufficient for existing amplification methods.

%%%%%%%%%%%%%%%%%%%%%%%%%%%%%%%%%%%%%%%%%%%%%%%%%%%%%%%

\subsection{Beyond weak-smallness: non-concentration for lifts}

In this paper we answer \cref{question1} in the negative for the sequence of Pitale lifts \cite{Pit05}, establishing QUE for these forms. 

The Pitale lift associates to each Hecke--Maass form of weight $1/2$ (in the Kohnen plus space) a Hecke--Maass form on a congruence quotient of hyperbolic $4$-space\footnote{In other words, it is a lift from the metaplectic group $\widetilde{\SL}_2$ to the spin group $\GSpin(1,4)$. By the Weyl law, the number of Pitale lifts with Laplace eigenvalue up to $\lambda$ is $\asymp \lambda$, while the total number of Hecke--Maass forms is $\asymp \lambda^2$.}. It is a non-holomorphic version of the Saito--Kurokawa lift described by Duke and Imamo\=glu \cite{DI96}, and provides some of the simplest counterexamples to the naive Ramanujan conjecture \cite{Pit05, CLPS91}. 

The analogous problem of equidistribution in the weight-aspect setting of Saito--Kurokawa lifts in $\Sp_4(\R)$ was considered by Cogdell and Luo \cite{CL11}. In that aspect, mass equidistribution was recently proved by J\"a\"asaari, Lester, and Saha \cite{JLS24} under the assumption of GRH, relying on period formulae (cf.\ \cite{LY14, BKY13, BC22}). As noted above, our approach is different in nature and unconditional, relying initially on measure rigidity tools (which are not applicable in the weight aspect).

Pitale lifts are non-tempered, so some of their Hecke eigenvalues are exceptionally large; however, the eigenvalues are still not large enough to exploit \eqref{eq:weak-smallness-intro2} using existing methods. To overcome this difficulty, we construct Hecke operators that simultaneously amplify a given spectral parameter and have an exceptionally small intersection with the relevant orbits, \emph{much smaller than expected from consideration of their individual components}. This construction forms the core of the paper.

More broadly, the existence of such operators demonstrates that temperedness (and in particular weak-smallness) is not a sharp condition for the applicability of the amplification method. In ongoing joint work with Simon Marshall, we intend to use this observation as a starting point to establish QUE for arbitrary sequences of Hecke--Maass forms on $\Gamma\bs\HH^4$. We also expect the ideas developed in this paper to be useful in other contexts, such as the sup-norm problem for Hecke--Maass forms.

%%%%%%%%%%%%%%%%%%%%%%%%%%%%%%%%%%%%%%%%%%%%%%%%%%%%%%%
 
\subsection{Statement of results}\label{sec:statement}

Consider the upper half-space model for hyperbolic $4$-space, given by
\begin{equation*}
    \mathbb{H}^4=\{z=(x_0, x_1, x_2, y)\in\R^4\mid y>0\}
\end{equation*}
with the metric $ds^2 = y^{-2}(dx_0^2+dx_1^2+dx_2^2+dy^2)$. Let $\Iso^+(\H^4)$ be the group of orientation-preserving isometries of $\H^4$. It is generated by translations $t_\beta$ and the inversion $s$ given by
\begin{align*}
    t_\beta:z\mapsto z+\beta \qquad \qquad \text{ and } \qquad \qquad s:z\mapsto -\frac{\overline{z}}{|z|^2},
\end{align*}
where $\beta\in V^3:=\{z\in \R^4\mid y=0\}$ and  $\overline{z} := (x_0,-x_1,-x_2,-y)$. We will let $\Gamma$ be a certain (congruence) group $\SV_2(\Z)$ which will be defined later. The important point for now is that its image in $\Iso^+(\H^4)$ is the non-uniform lattice generated by $s$ and $t_\beta$, for $\beta \in V^3 \cap \ZZ^4$. Denote the induced probability measure on $\Gamma\bs \H^4$ by $dz$.

In \cite{Pit05} Pitale constructed a sequence of Hecke--Maass cusp forms on $\Gamma\bs \HH^4$ lifted from Hecke--Maass cusp forms on $\HH^2$ of weight $1/2$ for $\Gamma_0(4)$. These are the main objects investigated in this paper, and we will review their key properties in \cref{sec:Pitale_lift}. Let $\{\psi_n\}_{n\geq 1}$ denote the sequence of all Pitale lifts, specified in \cref{def:Pitale_lifts}. They are ordered by Laplace eigenvalue and normalized by $\norm{\psi_n}_2 = 1$. Define the probability measures $\mu_n$ on $\Gamma\backslash \HH^4$ by
\begin{equation*}
    d \mu_n = \abs{\psi_n}^2 dz.
\end{equation*}
The following theorem is our main result. 

\begin{theorem}[QUE for Pitale lifts]\label{mainthm}
    The probability measures $\mu_n$ converge in the weak-$*$ topology to the uniform probability measure on $\Gamma\backslash \HH^4$.   
\end{theorem}

As in previous works on the subject, a microlocal lift construction reduces the problem to equidistribution of arithmetic quantum limits on $\Gamma\bs G$, where $G = \SV_2(\R)$ is a certain double cover of $\Iso^+(\H^4)$ (see \cref{sec:SV_2} for the definition). We will fix a certain diagonalizable subgroup $A$ of $G$, and call an $A$-invariant finite measure on $\Gamma\bs G$ an arithmetic quantum limit if it is the weak-$*$ limit of measures of the form $\abs{\phi_j}^2 dg$, where the $\phi_j$ are Hecke eigenfunctions on $\Gamma\bs G$ of unit $L^2$-norm and increasing Laplace eigenvalues.       

Two important inputs for \cref{mainthm} are the results of \cite{ShemTovSilberman:AQUE4d}, concerning the structure of arithmetic quantum limits on $\Gamma\bs G$, and our recent proof of non-escape of mass \cite{deFaveriShemTov:NonEscape4d}. These inputs reduce the problem to proving a non-concentration theorem, which is the new component of this paper and we now describe.  

Recall that $G = \SV_2(\RR)$ is a certain double cover of $\Iso^+(\H^4)$, and $\Gamma=\SV_2(\ZZ) \subset G$. Let $A$ be the group of real diagonal matrices in $G$. Let $M$ denote the compact part of the centralizer of $A$ in $G$, and $H$ denote the obvious embedding of $\SL_2(\CC)$ into $G$ when the latter is realized by a group of $2\times2$ matrices over the quaternions (see \cref{sec:SV_2} for details and concrete descriptions of each group). 

For each Pitale lift $\psi_n$ on $\Gamma \backslash \H^4$, let $\tilde{\psi}_n$ denote its ($L^2$-normalized) microlocal lift to $\Gamma\backslash G$ constructed by Silberman and Venkatesh in \cite{SilbermanVenkatesh:SQUE_Lift}.

\begin{theorem}[Non-concentration]\label{maintech}
    Every weak-$*$ limit $\tilde{\mu}$ of the measures $d\tilde{\mu}_n = |\tilde{\psi}_n|^2 \, dx$ on $\Gamma\backslash G$ satisfies 
    \begin{equation}\label{largeSets}
        \tilde{\mu}(\Gamma yHM)=0 
    \end{equation}
    for every $y\in\SV_2(\bar{\QQ}\cap\RR)$.  
\end{theorem}

It follows from \cite{ShemTovSilberman:AQUE4d} that concentration of measure on subsets of the form given in \eqref{largeSets} is the only remaining obstacle towards QUE in $\Gamma\backslash G$. The result above overcomes it for the Pitale lifts.
In \cref{sec:reduction} we promptly deduce \cref{mainthm} assuming \cref{maintech}. The rest of the paper is then devoted to the proof of \cref{maintech}. 

%%%%%%%%%%%%%%%%%%%%%%%%%%%%%%%%%%%%%%%%%%%%%%%%%%%%%%%

\subsection{Sketch of the proof}

\subsubsection{The parallel/transverse argument} \label{section:parallel-transverse}

To prove \cref{maintech} we apply the method of \cite{ShemTovSilberman:AQUE4d}, which reduces non-concentration on a proper submanifold $L \subset \Gamma\bs G$ to constructing Hecke operators with small $L^1$-norm (relative to their eigenvalues) when restricted to $L$. 

Let $U \subset L$ be a nice compact subset, and $U_\delta$ be its $\delta$-neighborhood. Call a Hecke translate $s.U$ \emph{parallel to $U$} if $\dim(s.U\cap U)=\dim(U)$, and otherwise call $s.U$ \emph{transverse}. If $\tau_j$ is a positive definite Hecke operator with $\tilde{\psi}_j$-eigenvalue $\lambda_j$, \cref{basic} gives an upper bound for $\lambda_j \cdot \tilde{\mu}_j(U_\delta)$ in terms of two quantities: the number of parallel Hecke translates of $U$, and the $\tilde{\mu}_j$-measure of tubes around the sets $s.U\cap U$ for $s.U$ transverse to $U$. This second quantity tends to $0$ as $\delta \to 0$ by induction on the dimension of $U$, where uniformity in $j$ can be arranged from control of the support of the $\tau_j$.

The task is then to find operators $\tau_j$ for which the number of parallel Hecke translates is small compared to $\lambda_j$. 
If $L$ has a \emph{small} stabilizer (see \cref{smallness-def}), then there always exist $\tau_j$ with the desired properties. In our case some of the stabilizers are conjugates of $H=\SL_2(\CC)$, which is \emph{not} small in $\SV_2(\RR)$. The key property of Pitale lifts that we use is that they are untempered, having especially large Hecke eigenvalues. However, the eigenvalues are in general not large enough for previous constructions of Hecke operators, exploiting sizes of stabilizers, to work.  

\subsubsection{Hecke operators avoiding the large subgroup $\SL_2(\CC)$} 

A crucial component of our proof of \cref{maintech} is the use of carefully constructed Hecke operators; they are given in \cref{sec:twoHeckeOperators} in terms of the two familiar Hecke operators $T_1(p)$ and $T_2(p)$. 

Our first observation is that for any orbit of $H$, the support of $T_2(p)$ does not intersect the orbit at all, for a positive proportion of primes $p$. This gives uniform bounds for the $L^2$-mass of $\tilde{\psi}_j$ near the orbit, as long as its $T_2(p)$-eigenvalue is large enough\footnote{This happens for a density one subsequence of the $\tilde{\psi}_j$, so using only $T_2(p)$ is enough to prove a version of quantum ergodicity (QE) for the (sparse) subsequence of Pitale lifts. However, there are still infinitely many lifts for which the $T_2(p)$-eigenvalues are too small.}. 

To handle the case where the $T_2(p)$-eigenvalue is small we use the operator $\sigma(p)=T_2(p)^2-(p+1)T_1(p)^2$. If $\tilde{\psi}_j$ has small $T_2(p)$-eigenvalue, we show that $\sigma(p)$ has the required amplification properties: spectrally its eigenvalue is forced to be large, and geometrically it was chosen so that certain coefficients in its decomposition as a linear combination of basic Hecke operators are small. In fact, we also need these amplification properties to hold for the positive definite operators $T_2(p)^2$ and $\sigma(p)^2$; we prove this by decomposing these operators explicitly as linear combinations of basic Hecke operators.   

To compute the decomposition we use a computer program. While these computations could in principle be presented here without referring to it, we believe it is more instructive to include our code for two reasons.
First, we compute the linear combination using the Satake transform, so most intermediate steps are quite algebraic and not particularly enlightening or conceptual (we do believe there is a conceptual explanation for the existence of appropriate Hecke operators, but do not focus on this point here). 
Second, the code may be of independent interest as it can be generalized and used to compute such decompositions in other reductive groups.

%%%%%%%%%%%%%%%%%%%%%%%%%%%%%%%%%%%%%%%%%%%%%%%%%%%%%%%

\subsection*{Acknowledgements}

We are grateful to Elad Zelingher for his valuable assistance with SageMath, and to Simon Marshall for his precise remarks correcting an inaccuracy in a previous version. We thank Elon Lindenstrauss for his support, constant encouragement, and fruitful discussions. We also thank Valentin Blomer, F{\'e}licien Comtat, Peter Humphries, Subhajit Jana, Philippe Michel, Zeev Rudnick, Peter Sarnak, and Lior Silberman for encouragement and helpful comments on an earlier draft of this paper.

%%%%%%%%%%%%%%%%%%%%%%%%%%%%%%%%%%%%%%%%%%%%%%%%%%%%%%%

\section{Reduction to non-concentration}\label{sec:reduction}

We now explain how \cref{mainthm} follows from \cref{maintech}. First we use the latter in combination with the results of \cite{ShemTovSilberman:AQUE4d} to obtain that arithmetic quantum limits on $\Gamma\bs G$ are proportional to the Haar measure. Then our non-escape of mass result \cite{deFaveriShemTov:NonEscape4d} upgrades proportionality to equality.

\begin{lemma}\label{lemma:H_conjugate_contained}
    Let $g\in G$ and suppose that $gHg^{-1}$ is defined over $\bar{\QQ}\cap\RR$ and contains $A$. Then $gHg^{-1}M\subset bHM$ for some $b\in\SV_2(\bar{\QQ}\cap\RR)$.
\end{lemma}

\begin{proof}
    Since $H$ and $gHg^{-1}$ are defined over $\bar{\QQ}\cap\RR$, it follows (see e.g.\ {\cite[Lemma 25]{ShemTovSilberman:AQUE4d}}) that there exists $b\in\SV_2(\bar{\QQ}\cap\RR)$ such that $bHb^{-1}=gHg^{-1}$. Thus 
    \begin{equation}\label{initial}
         gHg^{-1}M = bHb^{-1}M.
    \end{equation}
    
    Observe that $H$ contains both $A$ and $b^{-1}Ab$. These are two maximal split tori in $H$, hence conjugated by an element $h\in H$. That is, 
    $b^{-1}Ab=hAh^{-1}$. Thus $h^{-1}b^{-1}\in N_G(A)\subset HM$, so $b^{-1}\in HM$ and we get 
    \begin{equation*}
        bHb^{-1}M\subset bHM.  
    \end{equation*}
    The result follows from \eqref{initial}.
    
\end{proof}

\begin{proof}[Proof of \cref{mainthm}]
    We will show that every weak-$*$ limit point $\mu$ of the measures $d\mu_j = |\psi_j|^2\, dz$ is equal to the uniform probability measure $dz$ on $\Gamma\bs \HH^4$. 

    Indeed, it follows from \cite{SilbermanVenkatesh:SQUE_Lift} that the microlocal lifts $\tilde{\mu}_j$ on $\Gamma\backslash G$ have a weak-$*$ limit point $\tilde{\mu}$ which is $A$-invariant and projects to $\mu$. Applying \cite[Theorem 2]{ShemTovSilberman:AQUE4d}, it follows that $\tilde{\mu}$ is proportional to a convex combination of homogeneous measures. Each non-uniform component in this convex combination is the $\tilde{H}$-invariant measure on a subset $\Gamma ym\tilde{H}$, where $y\in\SV_2(\bar{\QQ}\cap\RR)$, $m\in M$, and $\tilde{H}$ is an algebraic subgroup of $G$ isomorphic to $\SL_2(\CC)$ such that $m\tilde{H}m^{-1}$ is defined over $\bar{\QQ}\cap\RR$.
    
    By \cref{lemma:H_conjugate_contained} and by the fact that all the copies of $\SL_2(\CC)$ in $\SV_2(\RR)$ are conjugated (see e.g.\ \cref{classification1}), each of the components in this convex combination is either the Haar measure $dg$ or has support contained in a set as in \eqref{largeSets}, of which there are only countably many. But $\tilde{\mu}$ gives zero measure to any of these sets by \cref{maintech}, hence $d\tilde{\mu}$ is proportional to $dg$. Therefore the projection $d\mu$ is proportional to $dz$. From \cite[Theorem 1]{deFaveriShemTov:NonEscape4d} it follows (in wide generality) that $\mu$ is a probability measure, so $d\mu = dz$.
    
\end{proof}

%%%%%%%%%%%%%%%%%%%%%%%%%%%%%%%%%%%%%%%%%%%%%%%%%%%%%%%

\section{Preliminaries}\label{sec:prelimineries}

\subsection{Hamilton quaternions}

Let $\bH$ denote the associative algebra of real Hamilton quaternions, with basis elements $1, \bi, \bj, \bk$ satisfying $\bi^2 = \bj^2 = \bk^2 = \bi\bj\bk = -1$. For $z \in \bH$, let $z^*$ denote the image of $z$ under the anti-isomorphism $z \mapsto z^*$ given by switching the sign of $\bk$. Also denote the usual quaternionic conjugation $z \mapsto \overline{z}$ given by switching the signs of $\bi, \bj, \bk$. The composition of these two is the isomorphism $z\mapsto z'$ given explicitly by switching the signs of $\bi,\bj$.  Let $N(z)$ denote the norm $N(z) = \abs{z}^2 := z \overline{z}$ and $\mathrm{Re}(z)$ denote the real part $\mathrm{Re}(z) := \frac{z + \overline{z}}{2}$.

\subsection{Isometries of \texorpdfstring{$\H^4$}{}}\label{sec:SV_2}
Let $C_3$ be the Clifford algebra generated by $1,\bi_1,\bi_2,\bi_3$ with the relations $\bi_r^2=-1$ and $\bi_r\bi_s=-\bi_s\bi_r$ for $r\ne s$. 
Let $\HH^{4}$ denote the upper half-space model for hyperbolic $4$-space,  
\begin{equation*}
    \HH^{4}=\{z=x_0+x_1 \bi_1+ x_2 \bi_2 + y \bi_3 \in C_3 \mid  y>0\}
\end{equation*}
with the hyperbolic metric $ds^2 = y^{-2}(dx_0^2+dx_1^2+dx_2^2+dy^2)$. Let $G = \SV_2(\RR)$ denote the group 
\begin{equation}\label{eq:SV_2_def}
    \SV_2(\R) = \left\{ \begin{pmatrix} a & b \\ c & d \end{pmatrix} \in M_2(\bH) \Bigm\vert \begin{pmatrix} a & b \\ c & d \end{pmatrix} \begin{pmatrix} 0 & 1 \\ -1 & 0 \end{pmatrix} \begin{pmatrix} a^* & c^* \\ b^* & d^* \end{pmatrix} = \begin{pmatrix} 0 & 1 \\ -1 & 0 \end{pmatrix} \right\}.
\end{equation}
It acts isometrically on $\HH^{4}$ by 
\begin{equation}\label{action1}
g\cdot z := (az+b)(cz+d)^{-1} \qquad \text{ for } z\in \H^4, 
\end{equation}
where we naturally identify the quaternions $\bH$ as a subalgebra of $C_3$ (by $\bi \mapsto \bi_1$ and $\bj \mapsto \bi_2$, so $\bk \mapsto \bi_1 \bi_2$). The action \eqref{action1} induces a short exact sequence
\begin{equation*}
    1 \to \{\pm 1\} \to SV_2(\R) \to \Iso^+(\HH^{4}) \to 1,
\end{equation*}
where ${\mathrm{Iso}^+(\HH^{4})}$ is the group of orientation-preserving isometries of $\H^4$. The hyperboloid model for $\HH^{4}$ gives also ${\mathrm{Iso}^+(\HH^{4})} \isom SO^+(1, 4)$ (the identity component of $\SO(1, 4)$) and $SV_2(\R) \isom \Spin(1, 4)$.
For more background on this see \cite{deFaveriShemTov:NonEscape4d} and references therein. 

\subsection{Subgroups of \texorpdfstring{$\SV_2(\RR)$}{}}

Let $A=\{\diag(a, a^{-1}) \mid a \in \R^\times \}$ denote the subgroup of real diagonal matrices in $G = \SV_2(\RR)$, and $M=\{\diag(\alpha, \alpha') \in M_2(\bH) \mid \alpha \bar{\alpha} = 1\}$ denote the compact part of its centralizer in $G$. Finally, let $H$ denote the obvious copy of $\SL_2(\CC)$ in $G = \SV_2(\RR)$. Viewing $\SV_2$ as an algebraic group scheme over $\ZZ$ via \eqref{eq:SV_2_def}, observe that $H$ corresponds to the real points of the algebraic subgroup given by constraining the coefficients of $\bj$ and $\bk$ to vanish.   

\subsection{The lattice \texorpdfstring{$\SV_2(\ZZ)$}{}}\label{sec:lattice}

It is clear from \eqref{eq:SV_2_def} how to view $\SV_2(\RR)$ as the group of $\RR$-points of an affine algebraic group $\SV_2$ over $\Z$. We consider the group of $\ZZ$-points $\SV_2(\ZZ)\subset \SV_2(\RR)$. Concretely, $\SV_2(\ZZ)$ is the subgroup of $\SV_2(\R)$ consisting of matrices whose entries belong to the Lipschitz integral quaternions 
\begin{equation*}
\bH(\Z) = \Z+\Z \bi+\Z \bj+\Z \bk.
\end{equation*}
Then $\SV_2(\ZZ)$ is a non-uniform lattice in $\SV_2(\R)$, and we will work over the quotient space\footnote{The image of $\Gamma = \SV_2(\Z)$ in $\Isom^+(\HH^4)$ is the group generated by the integral isometries described in the introduction, by \cite[Proposition 3.1]{Pit05}.} $\SV_2(\Z)\bs\HH^4$, with the probability measure $dz$ coming from the Riemannian measure $d\vol(z)= \frac{dx_0 \,dx_1 \,dx_2 \, dy}{y^4}$ on $\HH^4$.

\subsection {Maass forms on \texorpdfstring{$\H^4$}{}}\label{sec:cusp-forms}

With the choice of coordinates above, the Laplace--Beltrami operator on $\H^4$ is given by  
\begin{equation*}
    \Delta=y^2\left(\partial_{x_0}^2 + \partial_{x_1}^2 + \partial_{x_2}^2 + \partial_{y}^2\right) - 2 y\partial_y.
\end{equation*}

In our context, a Maass form $\phi$ on $\H^4$ is an $\SV_2(\Z)$-invariant eigenfunction of $\Delta$ satisfying the moderate growth condition $\abs{\phi(z)}\ll y^M$ as $y \to \infty$, for some constant $M$. Writing $\lambda=(\frac{3}{2})^2+r^2$ for the $\Delta$-eigenvalue, so that $\Delta \phi + \lambda \phi = 0$, it follows that $\phi$ has a Fourier expansion (see e.g.~\cite{Maa49})
\begin{equation*}
    \phi(z) = \phi_0(y) + y^{3/2} \sum_{0 \not= \beta \in V^3(\ZZ)} A(\beta) K_{\bi r}(2 \pi |\beta| y) e\left(\mathrm{Re}(\beta z)\right),
\end{equation*}
where $V^3(\Z) := \Z + \bi \Z + \bj \Z$, $K_{\bi r}$ denotes the $K$-Bessel function, and
\begin{equation*}
    \phi_0(y) = 
    \begin{cases}
        A_1 y^{3/2+\bi r} + A_2  y^{3/2-\bi r} & \text{ if } r \not= 0,\\
        A_1  y^{3/2} + A_2  y^{3/2}\log{y} & \text{ if } r = 0.
    \end{cases}
\end{equation*}
The form $\phi$ is called cuspidal, or a cusp form, if $A_1=A_2=0$, so that $\phi_0$ is identically zero.

\subsection {Hecke operators on \texorpdfstring{$\SV_2(\Z) \backslash \H^4$}{}}\label{subsection:Hecke_operators}

To define Hecke operators, consider the group of similitudes 
\begin{equation*}
    {\rm{GSV}}_2(\RR)^+=\left\{g=\begin{pmatrix}a&b\\ c&d \end{pmatrix}\in M_2(\bH) \Bigm\vert ad^*-bc^*=\mu(g) \in \R^+, ab^*, cd^*\in V^3\right\}.  
\end{equation*}
Like $\SV_2(\R)$, it acts on $\HH^4$ by M{\"o}bius transformations as in $\eqref{action1}$. Suppose that $g\in {\rm{GSV}}_2(\RR)^+$ belongs to the commensurator ${\rm{Comm}}_{{\rm{GSV}}_2(\RR)^+}(\SV_2(\ZZ))$ of $\SV_2(\Z)$ in ${\rm{GSV}}_2(\RR)^+$, i.e.\ that $\SV_2(\Z)\bs (\SV_2(\Z)g\SV_2(\Z))$ is finite. 
The corresponding Hecke operator $T$ on $L^2(\SV_2(\Z)\bs \H^4)$ is defined by 
\begin{equation}\label{eq:BasicHeckeOperatorDef}
    Tf(x)=\sum_{s\in S}f(sx),     
\end{equation}
where $S\subset \SV_2(\Z) g \SV_2(\Z)$ is a fixed set of representatives for $\SV_2(\Z)\bs \SV_2(\Z) g\SV_2(\Z)$. We will occasionally use the notation $T\sim g$.

To each prime number $p\ne2$ we shall associate two Hecke operators $T_1(p)$ and $T_2(p)$. To define them we follow the presentation of Pitale from \cite{Pit05} and choose, for each odd prime $p$, an element $\hat{\alpha}\in \bH(\ZZ)$ such that $N(\hat{\alpha}) = p$ and $p \nmid \hat{\alpha}^n$ for all $n\ge1$ (this is equivalent to $\mathrm{Re}(\hat{\alpha}) \neq 0$). 
Let $T_1(p)$ and $T_2(p)$ denote the operators corresponding to the diagonal matrices\footnote{Note that $T_1(p)$ is what Pitale calls $T_p$, and $T_2(p)$ is what Pitale calls $T_{p^2}$.} 
\begin{equation}\label{eq:T_p_matrices}
    T_1(p)\sim \begin{pmatrix}1&0\\0&p\end{pmatrix}\qquad \text{ and } \qquad T_2(p)\sim \begin{pmatrix}\hat{\alpha}&0\\0&\hat{\alpha}'p \end{pmatrix}.
\end{equation}

Let $\cH_p$ denote the $\CC$-algebra of operators generated by $T_1(p)$, $T_2(p)$, and the identity operator $I$. We have $T_1(p)\cdot T_2(p)=T_2(p) \cdot T_1(p)$, so that $\cH_p$ is commutative. Also, any element of $\cH_p$ commutes with any element of $\cH_q$ if $p$ and $q$ are distinct primes, so the $\C$-algebra generated by all the $\cH_p$ (for $p \neq 2$) is commutative. We denote this algebra (the \emph{Hecke algebra}) by $\cH$, and call its elements \emph{Hecke operators}. The Hecke operators, being defined by isometries, commute with $\Delta$. 
A cusp form $\phi$ on $\SV_2(\Z)\bs \H^4$ is called a {\it Hecke--Maass cusp form} if it is a joint eigenfunction of $\cH$ (thus a joint eigenfunction of $\Delta$ and $\cH$).

Fix an odd prime $p$. We call each Hecke operator in $\cH_p$ which corresponds to a single double coset as in \eqref{eq:BasicHeckeOperatorDef} a \emph{basic Hecke operator} (at $p$). The basic Hecke operators at $p$ can be parametrized as follows. Let $\tau_{m,l}(p)$ denote the operator corresponding to the diagonal matrix
\begin{equation}\label{eq:generalBasicHeckeOperators}
    \tau_{m,l}(p)\sim g_{m, l}(p) := \begin{pmatrix}{\hat{\alpha}}^m&0\\0&({\hat{\alpha}'})^mp^{l-m}\end{pmatrix}.
\end{equation}
In particular $T_1(p)=\tau_{0,1}(p)$ and $T_2(p)=\tau_{1,2}(p)$. 
By \cite[Lemma 5.1]{Pit05}, every basic Hecke operator at $p$ has the form $\tau_{m,l}(p)$ for $l, m \in \Z$ such that $l\ge 2m \geq 0$.

In the sequel we will consider certain Hecke operators which will be expressed explicitly as polynomials in $T_1(p)$ and $T_2(p)$. This will enable us to have useful expressions for their eigenvalues (in terms of Hecke-eigenvalues of half-integral weight forms, see \eqref{eq:Hecke_eigenvalue}). However, to analyze the \emph{geometric} properties of these operators, we will decompose them as linear combinations of the basic Hecke operators $\tau_{m,l}(p)$.  

\subsection{Pitale lifts}\label{sec:Pitale_lift}

In this paper we study the sequence of Hecke--Maass cusp forms $\psi_n$ on $\Gamma\backslash\H^4$ constructed by Pitale in \cite{Pit05}. Each $\psi_n$ is obtained as a lift from a Hecke--Maass cusp form $f_n$ on $\H^2$ of weight $1/2$ (in the Kohnen plus space) for $\Gamma_0(4)$. The $\psi_n$ are distinct, $L^2$-normalized, and ordered by Laplacian eigenvalue. We give a detailed definition of the $\psi_n$ and explain their basic properties in \cref{section:Pitale_appendix}. 

The only information we shall use about the lifts $\psi_n$ are their Hecke eigenvalues. If $\lambda_n(p)$ denotes the Hecke eigenvalue of $f_n$ for the operator $T_{p^2}$ defined in \eqref{eq:half_integral_Hecke_op_def}, then 
\begin{equation*}
    T_j(p) \psi_n = \Lambda_{j, n}(p) \cdot \psi_n
\end{equation*}
for $j\in \{1, 2\}$, where
\begin{equation}\label{eq:Hecke_eigenvalue}
    \Lambda_{1,n}(p)= p\big(p^{1/2} \lambda_n(p) + p + 1\big)\qquad \text{and} \qquad \Lambda_{2,n}(p)=(p+1)\big(p^{3/2} \lambda_n(p) + p - 1\big). 
\end{equation}

%%%%%%%%%%%%%%%%%%%%%%%%%%%%%%%%%%%%%%%%%%%%%%%%%%%%%%%

\section{Hecke operators}

In this section we describe the Hecke operators introduced in \cref{subsection:Hecke_operators} from a few different (but equivalent) perspectives, and use those descriptions to prove some of their basic properties.

\subsection{Hecke operators on \texorpdfstring{$\SV_2(\ZZ)\bs\SV_2(\RR)$}{}}

The Hecke operators $\cH$ can in fact be defined already on functions on the bundle $\SV_2(\ZZ)\bs \SV_2(\RR)$ of $\SV_2(\ZZ)\bs \HH^4$.
Indeed, in \eqref{eq:BasicHeckeOperatorDef} we defined a basic Hecke operator using single coset representatives for $\SV_2(\ZZ)\bs\SV_2(\ZZ)a\SV_2(\ZZ)$, where $a$ belongs to the commensurator of $\SV_2(\ZZ)$ in $\GSV_2(\RR)^+$. Since the center of $\GSV_2^+(\RR)$ acts trivially on $\HH^4$, we could equivalently take $a$ in the commensurator of $\SV_2(\ZZ)$ in $\SV_2(\RR)$. Then it is clear how to view the corresponding Hecke operator as an operator on $L^2(\SV_2(\ZZ)\bs \SV_2(\RR))$. On the other hand, consider the bijection
\begin{equation}\label{eq:ObviousBijection}
\SV_2(\ZZ)\bs\SV_2(\RR)\isom \PGSV_2(\ZZ)\bs \PGSV_2(\RR).
\end{equation}
Here 
\begin{equation*}
    \GSV_2(R):=\left\{ \begin{pmatrix}a&b\\c&d\end{pmatrix}\in M_2(\bH(R))\Bigm\vert ad^*-bc^*\in R^\times,\ ab^*, cd^*\in R+\bi R+\bj R\right\},
\end{equation*}
and $\PGSV_2(R)$ is the quotient of $\GSV_2(R)$ by its (group-theoretic) center, for every commutative ring with unit $R$. We set $\uG:=\PGSV_2$. Through \eqref{eq:ObviousBijection} we can identify the Hecke operators (and their eigenfunctions) on $L^2(\SV_2(\ZZ)\bs\SV_2(\RR))$ with those on $L^2(\uG(\ZZ)\bs\uG(\RR))$. The advantage in doing so is that the commensurator of $\uG(\ZZ)$ in $\uG(\RR)$ is $\uG(\QQ)$, so the basic Hecke operators can be represented by rational matrices (see example at the end of \cref{sec:HeckeOperatorsOnHomogeneousSpace}). We again use the notation $\tau_{m,l}(p)$ for the basic Hecke operator corresponding to the diagonal matrix given in \eqref{eq:generalBasicHeckeOperators}, acting on functions on $\SV_2(\ZZ)\bs\SV_2(\RR)$ (or equivalently those on $\uG(\ZZ)\bs\uG(\RR)$).

\subsection{The \texorpdfstring{$L^1$}{}-seminorm of a Hecke operator restricted to a subset}\label{sec:HeckeOperatorsOnHomogeneousSpace}

Let $\pi:\uG(\RR)\to \uG(\ZZ)\bs \uG(\RR)$ denote the quotient map.   
Each Hecke operator $\tau$ on $L^2(\uG(\ZZ)\bs\uG(\RR))$ has the form
\begin{equation}\label{generalHecke}
(\tau f)(x)=\sum_{s\in S}c_sf(sx),    
\end{equation}
where $S$ is a (finite) set of representatives for the image under $\pi$ of a finite union of double $\uG(\ZZ)$-cosets and $s\mapsto c_s$ is constant on each double coset. We call $\pi(S)\subset \uG(\ZZ)\bs \uG(\RR)$ the \emph{support} of $\tau$.

We also define, for any subset $U\subset \uG(\RR)$ and $1\leq p\leq \infty$, the seminorm
\begin{equation}\label{eq:HeckeNorm}
  \norm{\tau}_{L^p(U)}=\Big(\sum_{\substack{s\in S \\ \pi(s)\in\pi(U)}} \abs{c_s}^p \Big)^{1/p}.   
\end{equation}
For the case $U=\uG(\R)$ we shall simply write $\norm{\tau}_p$. We will only use this definition for $p=1, 2$, and $\infty$ (with the usual meaning for the latter). 

In light of \eqref{eq:ObviousBijection} we can also make the same definition for a subset $U\subset \SV_2(\RR)$, where now we consider a set of coset representatives $S \subset \SV_2(\RR)$. We will also use $\pi$ for the quotient map $\pi: \SV_2(\RR) \to \SV_2(\ZZ)\bs \SV_2(\RR)$, since the meaning is clear from context. For instance, the support of $T_1(p)$ when viewed as an operator on $L^2(\SV_2(\ZZ) \bs \SV_2(\RR)) = L^2(\Gamma \bs G)$ contains the coset $\Gamma \diag(\sqrt{p}^{-1},\sqrt{p})$; when viewed as an operator on $L^2(\uG(\ZZ)\bs\uG(\RR))$, it contains the coset $\uG(\ZZ)\diag(1,p)$. 

\subsection{Hecke operators from an \texorpdfstring{$S$}{}-arithmetic point of view}\label{sec:padicHeckeOperators}

For each finite set $S$ of odd primes, we have an isomorphism 
\begin{equation}\label{eq:AdelicRepresentation}
\uG(\ZZ)\bs\uG(\RR)\isom\Gamma_S\bs\uG(\RR)\times \prod_{p\in S}\uG(\QQ_p)/K_S    
\end{equation}
given by $\uG(\ZZ)g\mapsto \Gamma_S(g,1)K_S$, where $\Gamma_S=\uG(\ZZ[S^{-1}])$ acts diagonally from the left and $K_S=\prod_{p\in S}\uG(\ZZ_p)$ acts on $\prod_{p\in S}\uG(\QQ_p)$ from the right. We denote $K_p := \uG(\ZZ_p)$. 

Using \eqref{eq:AdelicRepresentation}, we may describe the local Hecke algebra $\cH_p$ as follows. Let $\cH(\uG(\QQ_p),K_p)$ be the convolution algebra of compactly supported $K_p$-bi-invariant functions on $\uG(\QQ_p)$, where convolution is defined with respect to the Haar measure $dg$ on $\uG(\QQ_p)$ normalized so that $K_p$ has measure $1$. Then every $\tau \in \cH(\uG(\QQ_p),K_p)$ defines via \eqref{eq:AdelicRepresentation} a linear operator on $L^2(\uG(\ZZ)\bs\uG(\RR))$ by
\begin{equation*}
    f \mapsto (f * \tau)(x) =\int_{\uG(\QQ_p)}\tau(g)f(xg^{-1})\, dg.
\end{equation*}
Since $\cH(\uG(\QQ_p),K_p)$ is commutative (this is a non-trivial fact), the set of all such linear maps is itself a commutative algebra of operators, which corresponds to $\cH_p$. For instance, each $\tau_{m, l}(p) \sim g_{m, l}(p) \in \GSV_2(\QQ)$ corresponds to $\tau = \1_{K_p g_{m, l}(p) K_p}$. The $\tau_{m, l}(p)$ are self-adjoint, as $\tau^* = \1_{K_p g_{m, l}(p)^{-1}K_p} = \tau$ (cf.\ \cite[Proposition 6.5.2]{RS07}).  From now on we will use the equivalent descriptions of $\cH_p$ without further comment. 

The following definition gives a convenient way to control the $L^1$-seminorm of a Hecke operator restricted to a subset, defined in \eqref{eq:HeckeNorm}.

\begin{definition}\label{def:L1-normAdelicIntersection}
    Let $\uV$ be an algebraic subgroup of $\uG$ over $\bar{\QQ}\cap\RR$ (that is, $\uV$ is a subgroup of the group $\uG_{\bar{\QQ}\cap\RR}$ obtained by extension of scalars). Suppose that $\uV$ is defined over $\QQ_p$ for an odd prime $p$. Given integers $l\geq 2m\geq 0$, we define  
    \begin{equation}\label{eq:DefinitionOfL}
    \cL_{m, l}(p,\uV):=|(\uV(\QQ_p)K_p \cap K_p g_{m, l}(p)K_p)/K_p|,
    \end{equation}
    where recall that $g_{m, l}(p) = \diag(\hat{\alpha}^m, (\hat{\alpha}')^m p^{l-m})$, so the Hecke operator $\tau_{m, l}(p)$ corresponds to the matrix $g_{m, l}(p)$ as in \eqref{eq:generalBasicHeckeOperators}. 
\end{definition}

\begin{remark} \ 
    \begin{enumerate}
        \item The statement that $\uV$ is defined over $\QQ_p$ means that there is a set of defining equations for $\uV$ (containing the defining equations for $\uG$) such that the number field generated by the coefficients of these equations can be embedded in $\QQ_p$. Fixing such an embedding allows us to make sense of $\uV(\QQ_p)$ and view it as a subgroup of $\uG(\Q_p)$.  
        The choice of the embedding is immaterial for our applications; we keep it implicit to avoid cumbersome notation.  
       
        \item The basic properties of $\cL_{m, l}(p,\uV)$ deduced in the sequel would also hold if we replaced the matrix $g_{m, l}(p)$ with any element of $\uG(\QQ_p)$. We prefer to work in this restricted setting for concreteness.  
    \end{enumerate}
\end{remark}

Observe that 
\begin{equation}\label{eq:L1NormsNonsense}
    \norm{\tau_{m, l}(p)}_{L^1(\uV(\RR))}\le \cL_{m, l}(p,\uV).
\end{equation}

\begin{lemma}[Bound is multiplicative for different primes]\label{lem:DifferentPrimesAreMultiplicative}
     Suppose $p \neq q$ are odd primes. Let $\uV$ be a subgroup of $\uG$ over $\bar{\QQ}\cap\RR$ which is defined over both $\QQ_p$ and $\QQ_q$. Then
     \begin{equation}\label{eq:multNorms}
        \norm{\tau_{m, l}(p)\tau_{m', l'}(q)}_{L^1(\uV(\RR))} \leq \cL_{m, l}(p,\uV) \cdot \cL_{m', l'}(q,\uV).
     \end{equation}
\end{lemma}
 
\begin{proof}
    We follow the argument in the proof of \cite[Lemma 3.8]{marshallupperboundsmaassforms}. Set $S = \{p, q\}$ and denote the left cosets corresponding to the operator $\tau_{m, l}(p)\tau_{m', l'}(q)$ by $\uG(\ZZ) \gamma_i$ with $\gamma_i \in \Gamma_S$. Each $\gamma_i^{-1}$ such that $\gamma_i \in \uG(\Z) \uV(\R)$ corresponds (embedding diagonally) to a distinct coset in $\mathcal{D} / \uG(\ZZ)$, where 
    \begin{equation*}
        \mathcal{D} = \Gamma_S \cap \big[ \uV(\R)\uG(\ZZ) \times K_p g_{m, l}(p)K_p \times K_q g_{m', l'}(q) K_q \big] \subset \uG(\R) \times \uG(\Q_p)\times \uG(\Q_q).
    \end{equation*}
    Here we used the self-adjointness $K_p g_{m, l}(p)^{-1}K_p = K_p g_{m, l}(p)K_p$ (and similarly for $q$).
    Since $\uV$ is defined over $\QQ_p$ and $\QQ_q$, we have a well-defined projection map
    \begin{equation*}
        \varphi:\mathcal{D} \to \big[ (\uV(\QQ_p)K_p \cap K_p g_{m, l}(p)K_p)/K_p \big] \times \big[ (\uV(\QQ_q)K_q \cap K_q g_{m', l'}(q)K_q)/K_q \big].
    \end{equation*}
    
    To finish the proof it is enough to show that $\varphi$ has fibers which are contained in the same right $\uG(\ZZ)$-coset. This is the case, since if $a_1, a_2 \in \mathcal{D}$ are mapped to the same pair of cosets $(g_pK_p, g_q K_q)$ then
    \begin{equation*}
        a_1^{-1} a_2 \in \Gamma_S\cap\big[\uG(\RR) \times K_p \times K_q\big] = \uG(\ZZ).
    \end{equation*}
    
\end{proof}

%%%%%%%%%%%%%%%%%%%%%%%%%%%%%%%%%%%%%%%%%%%%%%%%%%%%%%%

\section{Counting Hecke translates}\label{sec:HeckeTranslates}

Recall that $H\subset\SV_2(\RR)$ stands for the obvious copy of $\SL_2(\CC)$ in $\SV_2(\RR)$. 
The main result of this section is \cref{generalL1NormRestrictedToHEstimate}, which will be used to give an upper bound for $\norm{\tau}_{L^1(gHg^{-1})}$ when $\tau\in\cH_p$ for some $p$ which is good for $g$ (see \cref{goodPrimes} below). 

Let $\uH$ denote the subgroup of $\uG = \PGSV_2$ obtained by requiring that the coefficients of $\bj$ and $\bk$ in each entry vanish. Under the identification $\SV_2(\ZZ)\bs\SV_2(\RR)\isom \uG(\ZZ)\bs\uG(\RR)$ (see \eqref{eq:ObviousBijection}), the image of $H$ in $\SV_2(\ZZ)\bs\SV_2(\RR)$ is mapped to the image of $\uH(\RR)$ in $\uG(\ZZ) \bs \uG(\RR)$.

%%%%%%%%%%%%%%%%%%%%%%%%%%%%%%%%%%%%%%%%%%%%%%%%%%%%%%%

\subsection{Symplectic group and diagonal tori}

Throughout the paper we will use the fact that $\uG$ is split over $\QQ_p$ for every prime $p \neq 2$, and has a $\Q_p$-split maximal torus (of dimension $2$) consisting of (projections of) diagonal matrices. We will fix such a torus for each $p\ne2$ once and for all and call it the \emph{diagonal torus of $\uG(\Q_p)$}. To define it we use the following important isomorphism, which will also be useful for upcoming computations.

\subsubsection{Symplectic isomorphism}\label{section:symplectic_iso} We follow \cite[Proposition 6.3]{Pit05}, fixing some typos. Denote $J = \big(\begin{smallmatrix}
        0 & I_2 \\
        -I_2 & 0
    \end{smallmatrix}\big)$, and for any commutative ring with unit $R$ define the symplectic group
\begin{align*}
    \GSp_4(R) :=&\ \left\{ M \in M_4(R) \bigm\vert M J M^t = \mu(M) \cdot J \text{ for } \mu(M) \in R^\times \right\} \\
    =& \ \left\{ \begin{pmatrix}A&B\\C&D\end{pmatrix}\in M_4(R)\Bigm\vert A D^t - B C^t \in R^\times \cdot I_2, A B^t = B A^t, C D^t = D C^t  \right\}.
\end{align*}
For $p \neq 2$, fix $r, s \in \Z_p$ with $r^2+s^2 = -1$, which always exist. Then we can define a map $\psi_p : \bH(\Q_p) \to M_2(\Q_p)$ by sending $\alpha = a_0 + a_1 \bi + a_2 \bj + a_3 \bk$ to
\begin{equation*}
    \psi_p(\alpha) := 
    \begin{pmatrix}
        a_0 + a_1 r + a_2 s & a_3 - a_2 r + a_1 s \\
        -a_3 - a_2 r + a_1 s & a_0 - a_1 r - a_2 s
    \end{pmatrix}.
\end{equation*}
It can be checked directly that
\begin{itemize}
    \item $\psi_p$ is a $\Q_p$-algebra isomorphism,
    \item $\psi_p(\bH(\Z_p)) = M_2(\Z_p)$,
    \item $\det(\psi_p(\alpha)) = \abs{\alpha}^2$,
    \item $\psi_p(\alpha') = \abs{\alpha}^2 \cdot \psi_p(\alpha)^{-t}$, $\psi_p(\alpha^*) = \psi_p(\alpha)^{t}$, and $\psi_p(\overline{\alpha}) = \abs{\alpha}^2 \cdot \psi_p(\alpha)^{-1}$.
\end{itemize}
In particular, we obtain a group isomorphism $\psi_p : \GSV_2(\Q_p) \to \GSp_4(\Q_p)$ given by
\begin{equation*}
    \psi_p\left(\begin{pmatrix}
        \alpha & \beta \\ 
        \gamma & \delta
    \end{pmatrix}\right) := \begin{pmatrix}
        \psi_p(\alpha) & \psi_p(\beta) \\
        \psi_p(\gamma) & \psi_p(\delta)
    \end{pmatrix},
\end{equation*}
since the conditions $\alpha \delta^* - \beta \gamma^* \in \Q_p^\times, \alpha \beta^* = \beta \alpha^*, \gamma \delta^* = \delta \gamma^*$ defining $\GSV_2(\Q_p)$ exactly match those for $\GSp_4(\Q_p)$. Its restriction to integral points also gives $\GSV_2(\Z_p) \isom \GSp_4(\Z_p)$. Finally, taking a quotient by the respective centers (which consist of scalar matrices) implies the following.
\begin{lemma}
    For any prime $p \neq 2$, the map $\psi_p$ above gives isomorphisms $\uG(\Q_p) \isom \PGSp_4(\Q_p)$ and $\uG(\Z_p) \isom \PGSp_4(\Z_p)$.
\end{lemma}

By definition, for any commutative ring with unit $R$ we have
\begin{align*}
    \uH(R) := \left\{ \begin{pmatrix}a&b\\c&d\end{pmatrix} \in M_2(R[\bi])\Bigm\vert ad-bc \in R^\times \right\} \Big/ R^\times.
\end{align*}
The map $\psi_p$ gives embeddings $\uH(\Q_p)\hookrightarrow \PGSp_4(\Q_p)$ and $\uH(\Z_p)\hookrightarrow \PGSp_4(\Z_p)$. 

\subsubsection{Diagonal tori}\label{DiagTorus}

Consider the standard maximal torus 
\begin{equation*}
    T_p := \left\{ \diag(u_1, u_2, u_0 u_1^{-1}, u_0 u_2^{-1}) \bigm\vert u_0, u_1, u_2 \in \Q_p^\times \right\}
\end{equation*}
of $\GSp_4(\Q_p)$. Then we define the \emph{diagonal torus of $\uG(\Q_p)$} to be $T_{\uG}(\Q_p) := \Q_p^\times \cdot \psi^{-1}_p(T_p)$. The Weyl group of $(\GSp_4(\Q_p), T_p)$ is of order $8$, generated \cite[(6.6)]{Pit05} by 
\begin{equation}\label{eq:Weyl_group_gens}
    (u_0, u_1, u_2) \mapsto \{(u_0, u_2, u_1), (u_0, u_0 u_1^{-1}, u_2), (u_0, u_1, u_0 u_2^{-1}) \}.
\end{equation}
Similarly, we define the \emph{diagonal torus of $\uH(\Q_p)$} to be $T_{\uH}(\Q_p) := T_{\uG}(\Q_p) \cap \uH(\Q_p)$. Replacing $\Q_p$ with an arbitrary field extension in the definitions above, we obtain algebraic tori $T_{\uH}$ and $T_{\uG}$ defined over $\Q_p$.

\begin{lemma}[Diagonal torus for $\uH$]\label{lemma:H_diag_torus}
    For any prime $p \equiv 3 \pmod{4}$, we have $T_{\uH}(\Q_p) = \Q_p^\times \cdot \left\{ \diag(1, a) \bigm\vert a \in \Q_p^\times \right\}$ and $T_{\uH}$ is a maximal $\Q_p$-split torus of $\uH$.
\end{lemma}

\begin{proof}
    From the definition we directly compute
    \begin{equation*}
        \psi^{-1}_p(T_p) = \left\{ 
        \begin{pmatrix}
            \alpha & 0 \\
            0 & c \alpha'
        \end{pmatrix} 
        \Bigm\vert c \in \Q_p^\times, \alpha  = a + b(r\bi + s\bj) \text{ for } a, b \in \Q_p \text{ with } a\neq \pm b \right\} .
    \end{equation*}
    Since $r^2+s^2 = -1$ and $p \equiv 3 \pmod{4}$, we have $s \neq 0$, hence an element of $\psi^{-1}_p(T_p)$ projects to $\uH(\Q_p)$ if and only if $b = 0$. One then directly checks that $T_{\uH}$ is a maximal $\Q_p$-split torus of $\uH$ (again since $p \equiv 3 \pmod{4}$). 
    
\end{proof}

\subsubsection{Symplectic Hecke operators} For future reference, we compute the Hecke operators on $\PGSp_4(\Q_p)$ corresponding to $\tau_{m, l}(p) \sim g_{m, l}(p) = \diag(\hat{\alpha}^m, (\hat{\alpha}')^m p^{l-m}) \in \uG(\Q_p)$ under the isomorphism $\psi_p$.

\begin{lemma}[Dictionary for Hecke operators]\label{lemma:symplectic_Hecke_op}
    For any prime $p\neq 2$, we have
    \begin{equation*}
        \psi_p(\uG(\Z_p) g_{m,l}(p) \uG(\Z_p)) = \PGSp_4(\Z_p) \diag(1, p^m, p^l, p^{l-m}) \PGSp_4(\Z_p).
    \end{equation*}
\end{lemma}

\begin{proof}
    Observe that for $\hat{\alpha} \in \bH(\Z)$ with $|\hat{\alpha}|^2 = p$ and $p \nmid \hat{\alpha}^n$ for all $n\geq 1$, the Smith normal form over $\Z_p$ and the restricted isomorphism $\psi_p : \bH(\Z_p)^\times \to \GL_2(\Z_p)$ imply that for each $m \geq 0$ there exist $\beta_1, \beta_2 \in \bH(\Z_p)^\times$ with $|\beta_1 \beta_2|^2 = 1$ such that 
    \begin{equation*}
        \psi_p(\beta_1 \hat{\alpha}^m \beta_2) = \begin{pmatrix}
            1 &0\\
            0& p^m
        \end{pmatrix}.
    \end{equation*}
    Therefore $\psi_p((\beta_1 \hat{\alpha}^m \beta_2)') = |\beta_1 \hat{\alpha}^m \beta_2|^2 \cdot \psi_p(\beta_1 \hat{\alpha}^m \beta_2)^{-t} = \diag(p^m, 1)$. Thus denoting $\tilde{\beta}_j := \diag(\beta_j, \beta_j') \in \uG(\Z_p)$ for $j \in \{1, 2\}$ we have
    \begin{equation*}
        \psi_p(\tilde{\beta}_1 g_{m,l}(p) \tilde{\beta}_2) = \diag(1, p^m, p^l, p^{l-m}).
    \end{equation*}
    This finishes the proof, since $\psi_p(\uG(\Z_p)) = \PGSp_4(\Z_p)$.
    
\end{proof}

%%%%%%%%%%%%%%%%%%%%%%%%%%%%%%%%%%%%%%%%%%%%%%%%%%%%%%%

\subsection{Hecke translates in orbits of \texorpdfstring{$H$}{}} 

\begin{definition}[Good primes]\label{goodPrimes}
    For $g\in\SV_2(\bar{\QQ}\cap\RR)$, let $F\subset \RR$ be the smallest number field such that $g\in\SV_2(F)$. Call a prime $p$ \emph{good} for $g$ if $p\equiv 3 \pmod{4}$ and there is an embedding of $F$ into $\QQ_p$ under which $g$ is mapped into $\SV_2(\ZZ_p)$.
\end{definition}

Observe in particular that if $p$ is good for $g$ and $\uV$ is a rational subgroup of $\uG$, then $g\uV g^{-1}$ is defined over $\QQ_p$, and furthermore
\begin{equation}\label{eq:good_conjugate_group}
    \cL_{m, l}(p, g \uV g^{-1}) = \cL_{m, l}(p, \uV).
\end{equation}

\begin{lemma}[Positive proportion of good primes]\label{manyPrimes}
    Let $g\in\SV_2(\bar{\QQ}\cap\RR)$ and $\mathscr{P}_g$ be the set of primes that are good for $g$. There exists $\varepsilon_g>0$ such that $|\mathscr{P}_g \cap [P, 2P]| \geq \frac{\varepsilon_g P}{\log{P}}$ for every $P$ sufficiently large in terms of $g$.
\end{lemma}

\begin{proof}
    Recall that we fixed an embedding $\bar{\Q} \hookrightarrow \C$. We will work over the normal closure $L$ of $F(i)$. Let $c \in \Gal(L/\Q)$ denote complex conjugation. Since $F \subset \R$, it is fixed by $c$. Thus every unramified rational prime $p$ whose Frobenius is in the conjugacy class of $c$ satisfies $p \equiv 3 \pmod{4}$ and has a prime factor over $F$ of residual degree $1$. If $p$ is also sufficiently large (in terms of $g$), this shows that it is good for $g$. Therefore the lemma follows from the Chebotarev density theorem. 
\end{proof}

\begin{proposition}\label{noIntersectionLemma}
    Let $p$ be a prime which is good for $g \in \SV_2(\bar{\QQ}\cap\RR)$. Then 
    \begin{equation*}
        \cL_{m, l}(p,g\uH g^{-1})=0
    \end{equation*}
    for any $l \geq 2m > 0$.
\end{proposition}

Note that \cref{noIntersectionLemma} implies  
\begin{equation*}
    \norm{\tau_{m,l}(p)}_{L^1(gHg^{-1})}=0
\end{equation*}
whenever $l \geq 2m > 0$, due to \eqref{eq:L1NormsNonsense}. To recall the notation, see \eqref{eq:HeckeNorm} for the $L^1$-seminorm of a Hecke operator restricted to a set, \cref{def:L1-normAdelicIntersection} for $\cL_{m, l}(p, g\uH g^{-1})$, and \eqref{eq:generalBasicHeckeOperators} for $\tau_{m,l}(p)$. Here we view the latter as an operator on $L^2(\SV_2(\ZZ)\bs\SV_2(\RR))$.

\begin{proof}[Proof of \cref{noIntersectionLemma}]
    Let $p$ be a prime which is good for $g$ and $\varphi :F \to \Q_p$ be an embedding as in \cref{goodPrimes}, so $\varphi(g) \in \SV_2(\Z_p)$. 
    We first prove the statement for $g=1$. 
    The general case will follow by essentially the same argument, as we will see below.     
    
    Recall from \eqref{eq:generalBasicHeckeOperators} that $\tau_{m,l}(p)$ corresponds to the matrix $g_{m, l}(p) \in \uG(\Q_p)$. Since $p\equiv 3\pmod{4}$, we claim that 
    \begin{equation}\label{eq:empty_intersect}
        \uH(\QQ_p)K_p \cap K_p g_{m, l}(p) K_p = \emptyset.
    \end{equation}
    Indeed, otherwise apply the isomorphism $\psi_p: \uG(\Q_p) \to \PGSp_4(\Q_p)$ and the Cartan decomposition, which by \cref{lemma:symplectic_Hecke_op} shows that $\Q_p^\times\cdot \diag(1, p^m, p^l, p^{l-m})$ would belong to the orbit of an element of the maximal $\Q_p$-split torus $\psi_p(T_{\uH}(\Q_p))$ of $\psi_p(\uH(\Q_p))$ by the Weyl group of $\PGSp_4(\Q_p)$ with respect to its maximal torus $\psi_p(T_{\uG}(\Q_p))$ described in \cref{DiagTorus}. By \cref{lemma:H_diag_torus}, the elements of $\psi_p(T_{\uH}(\Q_p))$ are of the form $\Q_p^\times \cdot  \diag(1, 1, a, a)$ for $a\in \Q_p^\times$, and by \eqref{eq:Weyl_group_gens} the Weyl group acts by permuting the diagonal entries of $\Q_p^\times \cdot  \diag(1, 1, a, a)$. Therefore it cannot be mapped to $\Q_p^\times\cdot \diag(1, p^m, p^l, p^{l-m})$ for $l \geq 2m>0$, since they have a different number of distinct entries. Thus \eqref{eq:empty_intersect} holds, and the desired result follows for $g=1$.
    
    The case of general $g \in \SV_2(\bar{\QQ}\cap\RR)$ follows from \eqref{eq:good_conjugate_group}, which is a direct consequence of $\varphi(g) \in K_p$, since $p$ is good for $g$.
    
\end{proof}

\begin{proposition}\label{IntersectionWithH-estimate}
Let $p$ be a prime which is good for $g \in \SV_2(\bar{\QQ}\cap\RR)$. Then 
\begin{equation*}
    \cL_{0, l}({p,g\uH g^{-1}})\ll p^{2l}
\end{equation*}
for any $l\geq 0$, where the implied constant is absolute. 
\end{proposition}

The proof of \cref{IntersectionWithH-estimate} will be given in \cref{zlemSec}. For now, \cref{noIntersectionLemma} and \cref{IntersectionWithH-estimate} enable us to estimate $\norm{\tau}_{L^1(gHg^{-1})}$ for any Hecke operator $\tau\in\cH_p$, as long as $p$ is good for $g$ and we know the decomposition of $\tau$ as a linear combination of the basic operators $\tau_{m,l}(p)$. Namely, we have the following result. 

\begin{proposition}[General norm bound]\label{generalL1NormRestrictedToHEstimate}
Let $p$ be a prime which is good for $g \in \SV_2(\bar{\QQ}\cap\RR)$. Let $\tau\in\cH_p$ be of the form 
\begin{equation*}
    \tau=\sum_{l\geq 2m > 0} b_{m,l} \cdot \tau_{m,l}(p) +\sum_{l\geq 0}c_l \cdot \tau_{0,l}(p).
\end{equation*}
Then
\begin{equation*}
    \norm{\tau}_{L^1(gHg^{-1})}\ll \sum_{l\geq 0} \abs{c_l}p^{2l},
\end{equation*}
where the implied constant is absolute. 
\end{proposition}

\begin{proof}
    This follows directly from the triangle inequality combined with \eqref{eq:L1NormsNonsense}, \cref{noIntersectionLemma}, and \cref{IntersectionWithH-estimate}. 
    
\end{proof}

%%%%%%%%%%%%%%%%%%%%%%%%%%%%%%%%%%%%%%%%%%%%%%%%%%%%%%%

\subsection{Proof of \texorpdfstring{\cref{IntersectionWithH-estimate}}{}} \label{zlemSec}

In this section, consider $\uG$ and $\uH$ as $\QQ_p$-groups, for $p\equiv 3\pmod{4}$.
Let $T_{\uG}$ be the diagonal torus of $\uG$ defined in \cref{DiagTorus}. Let $T_{\uH}\subset T_{\uG}$ be the maximal $\QQ_p$-split torus of $\uH$, whose group of $\QQ_p$-points is given in \cref{lemma:H_diag_torus}. 

Let $G_p=\uG(\QQ_p)$, $H_p=\uH(\QQ_p)$, $K_p=\uG(\ZZ_p)$, and $K_p^{\underline{H}}=\uH(\ZZ_p)$. Denote
\begin{equation*}
    \rho_{\uG}=\frac{1}{2}\sum_{\alpha>0}\alpha,
\end{equation*}
where the sum runs over a choice of positive roots for $(\uG,T_{\uG})$. Analogously, let $\rho_{\uH}$ be one of the two roots\footnote{The root spaces have dimension $2$.} for $(\uH,T_{\uH})$, which we declare as the positive root. Denote the groups of cocharacters of $T_{\uG}$ and $T_{\uH}$ respectively by $X_*(T_{\uG})$ and $X_*(T_{\uH})$.

\subsubsection{Bounding intersections in terms of roots}

Following Marshall \cite{marshallupperboundsmaassforms}, for $\mu\in X_*(T_{\uG})$ we define 
\begin{equation*}
    \norm{\mu}^*=\max_{w\in W_{\uG}}\left<\rho_{\uG},w.\mu\right>,
\end{equation*}
where $W_{\uG}$ is the Weyl group of $(\uG,T_{\uG})$. Define $\norm{\cdot}_{\uH}^*$ for cocharacters of $T_{\uH}$ in the same way (taking maximum over the Weyl group of $(\underline{H}, T_{\underline{H}})$ and replacing $\rho_{\underline{G}}$ with $\rho_{\underline{H}}$).

\begin{remark}
Notice that we use the \emph{relative} root structure rather than the absolute root structure as was the case in \cite{marshallupperboundsmaassforms} and consequently in \cite{ShemTovSilberman:AQUE4d}. 
In \cref{small_groups_section} we will also recall the definition of smallness by using the \emph{absolute} root structure, where we view our algebraic groups as groups over an algebraically closed field.   
\end{remark}

This definition is convenient due to the following result.

\begin{lemma}[Volumes of double cosets]\label{VolLemma}
    For any cocharacter $\mu\in X_*(T_{\uG})$ we have 
    \begin{equation}\label{VolFormula}
    \abs{K_p\mu(p)K_p/K_p}\asymp p^{2\norm{\mu}^*}. 
    \end{equation}
    Analogously, for any cocharacter $\mu\in X_*(T_{\uH})$ we have
    \begin{equation*}
        \abs{K_p^{\uH} \mu(p) K_p^{\uH}/K_p^{\uH}} \asymp p^{2\norm{\mu}^*_{\uH}}.
    \end{equation*}
\end{lemma}

\begin{proof}
    This follows directly from the general result {\hspace{1sp}\cite[Proposition 1.6]{Casselman2012RemarksOM}}.  
    Indeed, let $\mathfrak{n}$ denote the unipotent subalgebra of the Lie algebra of $\uG$ generated by the root spaces for the positive roots, and let $\delta:=\abs{\det{\rm{Ad}_\mathfrak{n}}}_p$ denote the modulus character. Then $\delta(\mu(p))=p^{-2\left< \rho,\mu\right>}$. It is proved in \cite[Proposition 1.6]{Casselman2012RemarksOM} that if $\mu$ is positive in the sense that $\left<\alpha, \mu\right> \ge0$ for every positive root $\alpha$, then $\delta^{-1}(\mu(p))\asymp |K_p\mu(p)K_p/K_p|$. Since the Weyl orbit of any cocharacter $\mu$ contains a positive cocharacter, we get \eqref{VolFormula}. The same proof also applies verbatim for $\uH$.
    
\end{proof}

\begin{lemma}[Intersection bound]\label{MarshallForAllPrimes}
    For any cocharacter $\mu\in X_*(T_{\uG})$ we have
    \begin{equation*}
        \abs{(H_pK_p\cap K_p\mu(p)K_p)/K_p}\ll\sum_{\nu\in W_{\uG}.\mu\cap X_*(T_{\uH})}p^{2\norm{\nu}_{\underline{H}}^*}.
    \end{equation*}
\end{lemma}

\begin{proof}
   In \cite[Lemma 5.12]{marshallupperboundsmaassforms} the Cartan decomposition is used to show that 
    \begin{equation*}
    H_pK_p\cap K_p\mu(p) K_p= \bigcup_{\nu\in W_{\uG}.\mu\cap X_*(T_{\uH})}K_p^{\underline{H}}\nu(p)K_p.
    \end{equation*}
    Therefore the result follows from \cref{VolLemma}.
    
\end{proof}

\subsubsection{Explicit computation of $\norm{\cdot}^*_{\uH}$}

By \cref{lemma:H_diag_torus}, $T_{\uH}(\QQ_p)= \Q_p^\times \cdot \left\{ t_a = \diag(1, a) \bigm\vert a \in \Q_p^\times \right\}$. Let us compute the roots with respect to this maximal split torus. Observe that 
\begin{equation*}
    t_a\begin{pmatrix}x&y\\z&w\end{pmatrix}t_a^{-1}=\begin{pmatrix}x&ya^{-1}\\za&w\end{pmatrix},
\end{equation*}
so the (non-trivial) eigenvectors of $\Ad(t_a)$ on the Lie algebra of $\uH$ are $\left(\begin{smallmatrix}0&0\\u&0\end{smallmatrix}\right)$ with eigenvalue $a$ and $\left(\begin{smallmatrix}0&u\\0&0\end{smallmatrix}\right)$ with eigenvalue $a^{-1}$. Thus we have two roots: $t_a\mapsto a$, which we write additively as $e_1$, and $t_a\mapsto a^{-1}$, which we write additively as $-e_1$. Each eigenspace has dimension $2$, so declaring $e_1$ as the positive root we have $\rho_{\uH}=e_1$. The Weyl group has order $2$, its non-trivial element swapping $e_1$ and $-e_1$. 

If $\mu\in X_*(T_{\uH})$ then $\mu(p)=\diag(1, p^l)$ for some $l\in\ZZ$, so 
$\left<\rho_{\uH}, \mu\right>=l$. Suppose $w\in W_{\uG}$ is such that $w.\mu$ is again a cocharacter of $T_{\uH}$. By \eqref{eq:Weyl_group_gens}, the Weyl group acts by permuting the diagonal entries of $\Q_p^\times \cdot  \diag(1, 1, p^l, p^l) \in T_{\uG}(\Q_p)$, so $w.\mu = \mu$ or $-\mu$.

We are now ready to prove the main result of this section.

\begin{proof}[Proof of \cref{IntersectionWithH-estimate}]
By \eqref{eq:good_conjugate_group} it suffices to consider the case $g=1$.

Let $\mu \in X_*(T_{\uH})$ be the cocharacter corresponding to the Hecke operator $\tau_{0,l}(p)$, so that $\mu(p) = \diag(1, p^l)$. Then by the discussion above we have $\norm{w.\mu}^*_{\uH}=l$ whenever $w\in W_{\uG}$ is such that $w.\mu$ is a cocharacter of $T_{\uH}$. Thus the result follows from the bound of \cref{MarshallForAllPrimes} for $\cL_{0, l}(p, \uH)$. 

\end{proof}

%%%%%%%%%%%%%%%%%%%%%%%%%%%%%%%%%%%%%%%%%%%%%%%%%%%%%%%

\section{Decomposition of local amplifier}\label{sec:twoHeckeOperators}

The proof of \cref{maintech} uses certain carefully constructed Hecke operators in $\cH_p$. Their building blocks will be the operators $T_2(p)$ and 
\begin{equation}\label{eq:AlexOperator}
    \sigma(p):=T_2(p)^2-(p+1)T_1(p)^2.
\end{equation}  
Central to our approach are upper bounds for the $L^1(H)$-seminorms of $T_2(p)$ and $\sigma(p)$, as well as their squares  $T_2(p)^2$ and $\sigma(p)^2$.
The required bounds, as we shall see below, follow easily from the results of \cref{sec:HeckeTranslates} once we decompose these operators as linear combination of basic Hecke operators $\tau_{m,l}(p)$. To obtain these decompositions we used a computer program, which is presented in \cref{section:computation_appendix}.

\begin{lemma}[Decomposition into basic operators] \label{lem:DecompositionsUsingComputer} Let $p$ be an odd prime, and for simplicity denote $\tau_{m, l} = \tau_{m,l}(p)$. Then
    \begin{enumerate}[label=(\alph*)]
        \item
        \begin{align*}
            T_2(p)^2=(p^4+p^3+p^2+p)\tau_{0,0}+(p+1)\tau_{0,2}+(p-1)\tau_{1,2}+\tau_{2,4},
        \end{align*}
        
        \item 
        \begin{align*}
            \sigma(p)=-(p^3+p^2+p+1)\tau_{0,0}-(p^2+p+2)\tau_{1,2}+\tau_{2,4},
        \end{align*}
        
        \item
        \begin{align*}
            \sigma(p)^2 =\ &(2 p^8 + 4p^7 + 10p^6 + 15p^5 + 18p^4 + 17p^3 + 11p^2 + 6p + 1)\tau_{0,0} \\
            &+(2p^5 + 3p^4 + 6p^3 + 9p^2 + 8p + 4)\tau_{0,2} \\
            &- (2p^6 - 2p^5 - 11p^3 - 7p^2 - 6p) \tau_{1,2} +(p^2+p)\tau_{0,4} \\
            &-(2p^3+2p^2+4p)\tau_{1,4}+(2p^4 - 3p^3 + 3p^2 + 6)\tau_{2,4}\\
            & +(p-1)\tau_{2,6}-(2p^2 + p + 5)\tau_{3,6}+\tau_{4,8}.
        \end{align*}
    \end{enumerate}
\end{lemma}

\begin{proof}
    Given in \cref{section:computation_appendix}. 
    
\end{proof}

\begin{proposition}[Consequences for norms]\label{prop:IntersectionOfSpecialOperatorsWithH}
Let $p \neq q$ be primes which are good for $g\in\SV_2(\bar{\QQ} \cap \RR)$. Then 
\begin{align*}
    &\norm{T_2(p)}_{L^1(gHg^{-1})}=0, \qquad &\norm{T_2(p)^2}_{L^1(gHg^{-1})}\ll p^5, \qquad &\norm{T_2(p)T_2(q)}_{L^1(gHg^{-1})}=0, \\
    &\norm{\sigma(p)}_{L^1(gHg^{-1})}\ll p^3, \qquad &\norm{\sigma(p)^2}_{L^1(gHg^{-1})}\ll p^{10}, \qquad &\norm{\sigma(p)\sigma(q)}_{L^1(gHg^{-1})}\ll p^3q^3.
\end{align*}
The implied constants are absolute.
\end{proposition}

\begin{proof}
    This follows immediately from \cref{lem:DecompositionsUsingComputer}, \cref{generalL1NormRestrictedToHEstimate}, and \cref{lem:DifferentPrimesAreMultiplicative}.
    
\end{proof}

%%%%%%%%%%%%%%%%%%%%%%%%%%%%%%%%%%%%%%%%%%%%%%%%%%%%%%%

\section{Construction of global amplifier}

In this section we will utilize the Hecke operators of \cref{sec:twoHeckeOperators} and properties of their eigenvalues from \cref{sec:Pitale_lift} to disperse the $L^2$-mass of Pitale lifts on $\SV_2(\ZZ)\bs\SV_2(\RR)$ as in \cref{maintech}. For each Pitale lift $\psi_n$, recall the relations in \eqref{eq:Hecke_eigenvalue} for its Hecke eigenvalues.

\begin{lemma}\label{lem:GlobalAmplifier}
    Let $g\in\SV_2(\bar{\QQ} \cap \RR)$. Let $\mathscr{P}_g$ denote the set of primes which are good for $g$ (see \cref{goodPrimes}). For every $n\in \N$ and $P$ sufficiently large in terms of $g$, there exist coefficients $c_n(p) \in \{-1, 0, 1\}$ and Hecke operators $\tau_n(p)\in\{T_2(p), \sigma(p)\}$ such that writing
    \begin{equation*}
        \tau_n = \tau_n(P, g) := \Big(\sum_{\substack{P \leq p \leq 2P \\ p \in \mathscr{P}_g}} c_n(p) \tau_n(p)\Big)^2
    \end{equation*}
    and $\tau_n(P, g) \psi_n = \sL_n(P, g) \psi_n$, we have $\sL_n(P, g) > 0$ and
    \begin{equation*}
        \frac{\norm{\tau_n}_{L^1(gHg^{-1})}}{\sL_n(P, g)} \ll_{g} \frac{\log{P}}{P}.
    \end{equation*}
\end{lemma}

\begin{remark}
    The important point in the result above is that $\frac{\log{P}}{P} \to 0$ as $P \to \infty$. Observe also that $\tau_n(P, g)$ is positive definite, as $c_n(p) \in \R$ and the $\tau_n(p)$ are self-adjoint.
\end{remark}

\begin{proof}
    Denote $\mathscr{P}_g(P) := \mathscr{P}_g \cap [P, 2P]$. By \cref{manyPrimes}, there is $\varepsilon_g > 0$ such that $|\mathscr{P}_g(P)| \geq \frac{\varepsilon_g P}{\log{P}}$ for every $P$ sufficiently large in terms of $g$. Let
    \begin{equation*}
        \mathscr{A}_n(P) := \left\{ p\in \mathscr{P}_g(P) : |\lambda_n(p)| \leq \tfrac{1}{10}\right\}.
    \end{equation*}
    
    \noindent\textbf{Case 1:} $|\mathscr{A}_n(P)| < \frac{1}{2}\cdot |\mathscr{P}_g(P)|$.
    
    In this case, choose $\tau_n(p) = T_2(p)$. For each $p \in \mathscr{P}_g(P)$ we get $\tau_n(p)\psi_n = \Lambda_{2, n}(p)\psi_n$, where \eqref{eq:Hecke_eigenvalue} gives
    \begin{equation*}
        \Lambda_{2, n}(p) = (p+1)\big(p^{3/2}\lambda_n(p) + p-1\big).
    \end{equation*}
    Denote $\overline{\mathscr{A}_n}(P) := \mathscr{P}_g(P) \setminus \mathscr{A}_n(P)$. If $p \in \overline{\mathscr{A}_n}(P)$ and $P$ is sufficiently large, the display above implies $|\Lambda_{2, n}(p)| \geq \frac{1}{100}\cdot p^{5/2}$. Choosing $c_n(p) = 0$ if $p \in \mathscr{A}_n(P)$ and $c_n(p) = \sgn(\Lambda_{2, n}(p))$ if $p \in \overline{\mathscr{A}_n}(P)$, we conclude that
    \begin{equation*}
        \mathscr{L}_n(P, g) = \Big(\sum_{p \in \overline{\mathscr{A}_n}(P)} \left|\Lambda_{2, n}(p)\right| \Big)^2 \gg_g  \frac{P^7}{(\log{P})^2}.
    \end{equation*}
    
    Furthermore, by the triangle inequality we have
    \begin{equation*}
        \norm{\tau_n(P, g)}_{L^1(gHg^{-1})} \leq \sum_{p \in \overline{\mathscr{A}_n}(P)} \norm{T_2(p)^2}_{L^1(gHg^{-1})} + \sum_{p \ne q \in \overline{\mathscr{A}_n}(P)} \norm{T_2(p) T_2(q)}_{L^1(gHg^{-1})}.
    \end{equation*}
    Then \cref{prop:IntersectionOfSpecialOperatorsWithH} gives 
    \begin{equation*}
        \norm{\tau_n(P, g)}_{L^1(gHg^{-1})} \ll \frac{P^6}{\log{P}},
    \end{equation*}
    and the desired result follows.\\

    \noindent\textbf{Case 2:} $|\mathscr{A}_n(P)| \geq \frac{1}{2}\cdot |\mathscr{P}_g(P)|$.
    
    In this case, choose $\tau_n(p) = \sigma(p)$. For each $p \in \mathscr{P}_g(P)$ we get $\tau_n(p) \psi_n = \Lambda_{\sigma, n}(p)\psi_n$, where \eqref{eq:AlexOperator} implies
    \begin{equation*}
        \Lambda_{\sigma, n}(p) = \Lambda_{2, n}(p)^2 - (p+1)\Lambda_{1, n}(p)^2.
    \end{equation*}
    If $p \in \mathscr{A}_n(P)$ and $P$ is sufficiently large, then by \eqref{eq:Hecke_eigenvalue} we obtain
    \begin{equation*}
         \Lambda_{1,n}(p)^2= p^2\big(p^{1/2} \lambda_n(p) + p + 1\big)^2 \geq p^2\big(p - p^{1/2}\big)^2 \geq \frac{p^4}{2}
    \end{equation*}
    and \begin{equation*}
        \Lambda_{2,n}(p)^2=(p+1)^2\big(p^{3/2} \lambda_n(p) + p - 1\big)^2 \leq (p+1)^2\big(\tfrac{1}{10}\cdot p^{3/2} + p - 1\big)^2 \leq \frac{p^5}{50},
    \end{equation*}
    which combine into
    \begin{equation*}
        \Lambda_{\sigma, n}(p) \leq \frac{p^5}{50} - (p+1)\cdot\frac{p^4}{2} \leq -\frac{p^5}{3}.
    \end{equation*}
    
    Choosing $c_n(p) = 0$ if $p \in \overline{\mathscr{A}_n}(P)$ and $c_n(p) = -1$ if $p \in \mathscr{A}_n(P)$, we conclude that
    \begin{equation*}
        \mathscr{L}_n(P, g) = \Big(\sum_{p \in \mathscr{A}_n(P)} \left|\Lambda_{\sigma, n}(p)\right| \Big)^2 \gg_g  \frac{P^{12}}{(\log{P})^2}.
    \end{equation*}
    
    As before, the triangle inequality yields
    \begin{equation*}
        \norm{\tau_n(P, g)}_{L^1(gHg^{-1})} \leq \sum_{p \in \mathscr{A}_n(P)} \norm{\sigma(p)^2}_{L^1(gHg^{-1})} + \sum_{p \ne q \in \mathscr{A}_n(P)} \norm{\sigma(p) \sigma(q)}_{L^1(gHg^{-1})}.
    \end{equation*}
    Another application of \cref{prop:IntersectionOfSpecialOperatorsWithH} then gives 
    \begin{equation*}
        \norm{\tau_n(P, g)}_{L^1(gHg^{-1})} \ll \frac{P^{11}}{\log{P}} + \frac{P^8}{(\log{P})^2} \ll\frac{P^{11}}{\log{P}},
    \end{equation*}
    and we once again obtain the desired result.
    
\end{proof}

%%%%%%%%%%%%%%%%%%%%%%%%%%%%%%%%%%%%%%%%%%%%%%%%%%%%%%%

\section{Classification of subgroups} \label{classificationSec}

Recall that $G=\SV_2(\RR)$, $H=\SL_2(\CC)$, and $M=\{\diag(\alpha, \alpha') \in M_2(\bH) \mid \alpha \bar{\alpha} = 1\}$. In this section we aim to classify the connected real algebraic subgroups of $G$ contained in $HM$. The main result is \cref{finalClassification}, which says that each of these subgroups is either small, in a sense made precise below, or is contained in a conjugate of $H$. 

\subsection{Small subgroups}\label{small_groups_section}

View $\SV_2$ as a $\bar{\QQ}$-group. Fix a reductive subgroup $\underline{L}$ of $\SV_2$ (defined over $\bar{\QQ}$). Fix $\underline{T}_{\underline{L}} \subset\underline{T}_{\SV_2}$ maximal tori of $\underline{L}$ and $\SV_2$, and let $W_{\underline{L}} ,W_{\SV_2}$ denote the corresponding Weyl groups and $\Phi_{\underline{L}}, \Phi_{\SV_2}$ the corresponding root systems. Fix a notion of positivity and let $\rho_{\underline{L}}, \rho_{\SV_2}$ be the corresponding half-sums of positive roots of $\Phi_{\underline{L}}$ and  $\Phi_{\SV_2}$.
Following \cite{marshallupperboundsmaassforms}, for any cocharacter $\mu$ of $\underline{T}_{\SV_2}$ we define
\begin{equation*}
    \norm{\mu}^*_{\SV_2}=\max_{w\in W_{\SV_2}}\left<\rho_{\SV_2},w.\mu\right>,
\end{equation*}
and analogously define $\norm{\cdot}^*_{\underline{L}}$ for cocharacters of $\underline{T}_{\underline{L}}$ (where we take $w \in W_{\underline{L}}$ and replace $\rho_{\SV_2}$ with $\rho_{\underline{L}}$). 
Inspired by \cite{marshallupperboundsmaassforms}, we recall the following notions of smallness, which were previously used in \cite{ShemTovSilberman:AQUE4d} in the context of QUE. 

\begin{definition}[Smallness] \label{smallness-def} \
\begin{enumerate}
    \item We say that $\underline{L}$ is \emph{tempered} in $\SV_2$ if for every cocharacter $\mu$ of $\underline{T}_{\underline{L}}$ we have 
    \begin{equation*}
        \norm{\mu}^*_{\underline{L}}\le\frac{1}{2}\norm{\mu}^*_{\SV_2}.
    \end{equation*}
    
    \item We say that $\underline{L}$ is \emph{weakly small} in $\SV_2$ if it is tempered in $\SV_2$ and either $\dim \underline{T}_{\underline{L}} < \dim \underline{T}_{\SV_2}$, or there exists a cocharacter $\mu$ of $\underline{T}_{\underline{L}}=\underline{T}_{\SV_2}$ such that 
    \begin{equation*}
        \max_{w\in W_{\SV_2}}\norm{w.\mu}_{\underline{L}}^* < \frac{1}{2} \norm{\mu}^*_{\SV_2}.
    \end{equation*}
    
    \item We say that an abstract subgroup $S\subset \SV_2(\RR)$ is \emph{small} if there exists a weakly small subgroup $\underline{P}$ of $\SV_2$, defined over $\bar{\QQ}\cap\RR$, such that $S\subset \underline{P}(\RR)$. Call $S$ \emph{virtually small} if it has a finite-index small subgroup.
\end{enumerate}    
\end{definition}

\subsection{Subgroups of \texorpdfstring{$G$}{} contained in \texorpdfstring{$HM$}{}}

Let $K$ denote the maximal compact subgroup of $G$ given by
\begin{equation*}
    K=\left\{\begin{pmatrix} a & b \\ -b' & a' \end{pmatrix} \in M_2(\bH) \Bigm\vert \abs{a}^2+\abs{b}^2=1, ab^* \in V^3\right\},
\end{equation*}
and let $N$ denote the unipotent subgroup of $G$ given by 
\begin{equation*}
    N = \left\{n(u) := \begin{pmatrix} 1 & u \\ 0 & 1 \end{pmatrix} \Bigm\vert u \in V^3\right\},
\end{equation*}
where we recall that $V^3 = \R + \R \bi + \R \bj$.

\begin{lemma}\label{smallDimUnipotentLemma}
    Suppose $U\subset G$ is a unipotent subgroup (that is, closed and conjugate to a subgroup of $N$) and $U$ is contained in the set $HM$. Then $U$ has dimension at most $2$.   
\end{lemma}

\begin{proof}
    We may assume $U$ is connected.
    Let $g=\begin{pmatrix}a&b\\c&d\end{pmatrix}\in G$ be such that $gUg^{-1}\subset N$. 
    Then $gUg^{-1}\subset gHMg^{-1}$. We will show that the set of all $u$ such that 
    \begin{equation}\label{boundUnipotentDim}
    n(u)g\in gU,
    \end{equation}
    viewed as a subspace of $V^3$, has dimension at most $2$. 
    For this notice first that \eqref{boundUnipotentDim} implies, by replacing $u$ by $tu$ and differentiating at $t=0$,  
    \begin{equation}\label{boundUnipotent2}
        \begin{pmatrix}0&u\\0&0\end{pmatrix}g=g(X+Y),
    \end{equation}
    where $X$ and $Y$ belong to the Lie algebras of $H$ and $M$ respectively, that is $X=\begin{pmatrix}x&y\\z&w\end{pmatrix}$ where $x,y,z,w\in\CC$ and $Y=\diag(\alpha,\alpha')$ where $\alpha$ is a purely imaginary quaternion (in fact, $w=-x$ but we don't use it here). By computing the upper right entry of both sides of \eqref{boundUnipotent2} we get 
    \begin{equation}\label{upperRight}
    ud=ay+bw+b\alpha'. 
    \end{equation}
    By computing the lower right entry of \eqref{boundUnipotent2} we get
    \begin{equation}\label{lowerRight}
        0=cy+dw+d\alpha'.
    \end{equation}
    We assume without loss of generality that $d\ne0$ as otherwise we shall use the equations obtained by comparing the first column of \eqref{boundUnipotent2}. Thus $w+\alpha'=-d^{-1}cy$, due to \eqref{lowerRight}, and substituting this in \eqref{upperRight} we get $u=(a-bd^{-1}c)yd^{-1}$. Thus $u$ belongs to the image of $\CC$ under a real linear map, and the space of such $u$ is at most $2$-dimensional.   

\end{proof}

\begin{corollary}\label{smallUnipotentLemma}
    Suppose $L\subset gHM$ is a subvariety, and $S\subset G$ is the stabilizer of $L$. Then any unipotent subgroup of $S$ has dimension at most $2$.
\end{corollary}
\begin{proof}
    Fix $l=ghm\in L$, where $h\in H$ and $m\in M$. 
    Let $s\in S$. Then $sghm\in gHM$, so $h^{-1}g^{-1}sgh\in HM$. Thus $S$ is contained in a conjugate of $HM$ and the result follows from \cref{smallDimUnipotentLemma}. 
    
\end{proof}

\begin{lemma}\label{classification1}
Suppose that $H'\subset G$ is a subgroup isomorphic to $H$ as algebraic groups over $\RR$.
Then $H'$ is conjugate to $H$ inside $G$. 
\end{lemma}

\begin{proof}
    Since maximal tori are conjugate to each other we may assume that $H'$ contains the maximal split torus $A=\{\diag(a,a^{-1})\mid a\in\RR^\times\}$. Since $H'$ is isomorphic to $H$ we have the root groups with respect to $A$ and those must lie inside the unipotent subgroups $N^+$ and $N^-$ of upper and lower triangular matrices with $1$'s on the diagonal. Suppose $U$ is a maximal unipotent subgroup of $H'$ contained in $N^+$. We can conjugate it to $N^+\cap H$ by using an element of the centralizer of $A$. Thus without loss of generality $H'$ contains $A$ and $N^+\cap H$. Since the normalizer of $A$ in $H'$ must also normalize $N^+\cap H$
    and also be two-dimensional, we conclude that $H'$ contains all the diagonal complex matrices in $G$. It is easy to see also that $N^-\cap H$ is contained in $H'$, so we are done.  
    
\end{proof}

\begin{lemma}\label{classification2}
    There is no maximal compact subgroup of $G$ contained in the set $HM$. 
\end{lemma}

\begin{proof}
    Since all maximal compact subgroups are conjugate to each other, it is enough to show that for every $g\in G$, only a proper subspace of the Lie algebra of $K$ is contained in the set $g(\mathfrak{h}+\mathfrak{m})g^{-1}$. A similar computation to the one in the proof of \cref{smallDimUnipotentLemma} shows that the maximal dimension of a subspace of the space of matrices of the form $\begin{pmatrix}0&b\\-\bar{b}&0\end{pmatrix}$, contained in $g(\mathfrak{h}+\mathfrak{m})g^{-1}$, is $2$. Since the Lie algebra of $K$ clearly contains the space of matrices of this form, which is $3$-dimensional, we are done.  
    
\end{proof}

\begin{lemma}\label{finalClassification}
    Let $S$ be a connected real algebraic subgroup of $G$ and suppose $S\subset HM$. Then either $S$ is contained in a conjugate of $H$, or $S$ is conjugate to a small group (see \cref{smallness-def}).
\end{lemma}

\begin{proof}
    We use the Levi decomposition to write $S=S_l\ltimes S_u$ where $S_l$ is reductive and $S_u$ is unipotent and divide into cases:
    \begin{enumerate}
    \item \emph{$S_u$ is trivial.} 
    Write $S=S_0\times T_0$ where $S_0$ is semisimple and $T_0$ is a torus (up to finite index) and split into subcases:
    \begin{enumerate}
        \item \emph{$\dim T_0=0$.} In this case $S$ is semisimple. If $S$ is compact then it is a proper subgroup of a maximal compact subgroup, due to \cref{classification2}, and we are done by \cite[Lemma 27]{ShemTovSilberman:AQUE4d}. Otherwise, suppose $S$ contains a copy of $\SL_2(\CC)$. Then the other direct factors of $S$ have dimension at most $2$, so from the assumption that $S$ is semisimple we get that $S$ is isomorphic to $\SL_2(\CC)$, and we are done by \cref{classification1}. Thus we are left with the case where $S$ is isogeneous to either $\SL_2(\RR)$ or $\SL_2(\RR)\times \SO(3)$. The case where $S$ is isogeneous to $\SL_2(\R)$ was covered in \cite[Lemma 27]{ShemTovSilberman:AQUE4d}, so we may assume that $S$ is isogeneous to $\SL_2(\RR)\times \SO(3)$. Without loss of generality the copy of $\SL_2(\R)$ in this direct product is the standard copy in $\SV_2(\RR)$, and since its centralizer is a one dimensional compact torus we get a contradiction.  
        
        \item \emph{$\dim T_0=1$.} Assume first that $T_0$ is compact. If $S_0$ is compact then $S$ is compact and we are done. If $S_0$ is not compact, then it has a maximal one dimensional split torus, so it is isogeneous to $\SL_2(\RR)$. Thus $S$ is isogeneous to $\SL_2(\RR)\times S^1$ and this case was covered in \cite[Lemma 27]{ShemTovSilberman:AQUE4d}. If $T_0$ is not compact then without loss of generality $T_0=A$ and $S_0=M$, and this case was covered in \cite[Lemma 27]{ShemTovSilberman:AQUE4d}. 
        
        \item \emph{$\dim T_0=2$.} In this case $S_0$ is compact. On the other hand $S_0$ cannot contain a non-trivial torus, so $S_0$ is trivial and we are done. 
    \end{enumerate}
    
    \item \emph{$S_u$ is $1$ dimensional.} This case was covered in \cite[Lemma 27]{ShemTovSilberman:AQUE4d}. 
    
    \item \emph{$S_u$ is $2$ dimensional.} We may assume without loss of generality that $S_u$ is the subgroup of $G$ of upper triangular matrices with $1$'s on the diagonal and a complex number in the upper right entry. The connected component of the normalizer of this group is $M_1AN$, where $M_1$ is a proper subgroup of $M$, hence a torus, so it must be the compact diagonal subgroup of $\SL_2(\CC)$. Thus $S\subset M_1AN$. Since $S_l$ is reductive and is contained in $M_1AN$, which is solvable, it is in fact a torus and thus may be conjugated by an element $n\in N$ into the maximal torus $M_1A$, that is $nS_ln^{-1}\subset M_1A$. But $nS_u n^{-1}=S_u$, since $S_u\subset N$ and $N$ is abelian. Thus $nSn^{-1}\subset M_1AS_u\subset H$ and so we are done.  
\end{enumerate}
    
\end{proof}

%%%%%%%%%%%%%%%%%%%%%%%%%%%%%%%%%%%%%%%%%%%%%%%%%%%%%%%

\section{Proof of \texorpdfstring{\cref{maintech}}{}}

Recall that for any subset $U \subset G$, the dimension $\dim U$ of $U$ is defined as the (topological) dimension of the Zariski closure $\bar{U}^Z$ of $U$ (viewing $U$ as a subset of the real variety $G$). We use the following key definition from \cite{ShemTovSilberman:AQUE4d}. 

\begin{definition}(Transverse and parallel)\label{def:parallel}
     We say that $s\in G$ is \emph{transverse} with respect to a subset $U\subset G$ if $\dim(\gamma sU\cap U) < \dim U$ for all $\gamma\in\Gamma$. Otherwise, call $s$ \emph{parallel} with respect to $U$. 
\end{definition} 

\begin{lemma}\label{lemma:parallel_equals_stabilizer}
    Let $U \subset G$ be a subset with irreducible Zariski closure $\bar{U}^Z$. Then the set of elements of $\SV_2(\bar{\QQ}\cap\RR)$ which are parallel with respect to $U$ is contained in $\Gamma S$, where 
    \begin{equation*}
        S = \{ s \in \SV_2(\bar{\QQ}\cap\RR) \mid s \bar{U}^Z = \bar{U}^Z \}
    \end{equation*}
    is the stabilizer of $\bar{U}^Z$ in $\SV_2(\bar{\QQ}\cap\RR)$.
\end{lemma}

\begin{proof}
    An element $s' \in \SV_2(\bar{\Q}\cap \R)$ is parallel with respect to $U \subset G$ if and only if there exists $\gamma \in \Gamma$ such that $\dim(\gamma s'U\cap U) = \dim U$. In that case, since $\bar{U}^Z$ is irreducible by hypothesis, we must have $\overline{\gamma s'U\cap U}^Z = \bar{U}^Z$. But $\overline{\gamma s'U\cap U}^Z \subset \gamma s'\bar{U}^Z\cap \bar{U}^Z$, hence $\gamma s'\bar{U}^Z = \bar{U}^Z$, again by irreducibility. Thus denoting $s := \gamma s' \in S$, we obtain $s' = \gamma^{-1} s \in \Gamma S$, as desired.

\end{proof}

As in \cite{ShemTovSilberman:AQUE4d}, a key idea for our proof is distinguishing between parallel and transverse Hecke returns. This is encapsulated in the following result, where $U_\delta$ denotes the closed $\delta$-neighborhood of $U$ (with respect to any fixed metric compatible with the analytic structure of $G$ -- the exact choice of the metric is immaterial).

\begin{proposition}[{\hspace{1sp}\cite[Lemma 17]{ShemTovSilberman:AQUE4d}}] \label{basic}
    Let $L$ be a subvariety of $G$ and $U=U_0\subset L$ a Zariski dense subset of $L$. 
    Assume that $U$ is compact in the analytic topology and is contained in a fundamental domain $\cF$ for
    $\Gamma\bs G$.  
    
    Let $\tau$ be a positive definite Hecke operator and $S\subset G$ be a set of representatives
    for $\supp(\tau)$, which we divide into parallel and transverse elements (with respect to $U$) as $S=\mathcal{P}\sqcup \cT$.
    Then there exists a finite collection $\cN$ of subsets of $U_0$, depending only on $\cT$ and $U$, such that each $N\in\cN$ is of the form $N=U_0\cap b U_0$ with $b\in\SV_2(\bar{\QQ}\cap\RR)$, satisfies $\dim N < \dim L$, and the following holds.
    For every $\delta>0$ there exists $0 < \delta' < \delta$ such that for any eigenfunction $\phi$ of $\tau$
    with eigenvalue $\lambda>0$, we have 
    \begin{equation*}
    \mu_\phi(U_{\delta'})\le
     \frac{1}{\lambda}\left(\norm{\tau}_{L^1(\mathcal{P}\cap U_{\delta'}(U_{\delta'})^{-1})}
    +\frac{\norm{\tau}_\infty}{\mu_\phi(U_{\delta'})}\sum_{N\in\cN}\mu_\phi(N_\delta)\right).
    \end{equation*}
\end{proposition} 

We also need the following classification of stabilizers, which will be a consequence of the results of \cref{classificationSec}.

\begin{lemma}[Classification of stabilizers]\label{stabilizerscor}
    Let $y\in\SV_2(\bar{\QQ}\cap\RR)$ and consider a nonempty subset $U\subset yHM\subset G$. Then the stabilizer of $U$ in $\SV_2(\bar{\QQ}\cap\RR)$ is either virtually small (see \cref{smallness-def}), or is virtually contained in a conjugate of $H$ by an element of $\SV_2(\bar{\QQ}\cap\RR)$ (that is, the stabilizer contains a finite index subgroup contained in such a conjugate of $H$). 
\end{lemma}

For the proof we use the following general fact. 

\begin{lemma}\label{def_over_field}
    If $V$ is a real algebraic variety that contains a Zariski dense subset of $(\bar{\QQ} \cap \RR)$-points, then $V$ is defined over $\bar{\QQ}\cap\RR$.     
\end{lemma}

\begin{proof}
    This follows directly from fixing a basis $\{\alpha_i\}_i$ for $\R$ over $(\bar{\Q}\cap \R)$ and writing each defining polynomial $f$ of $V$ as a finite linear combination $f = \sum_i \alpha_i f_i$ of polynomials $f_i$ with coefficients in $\bar{\Q}\cap \R$.
    
\end{proof}

\begin{proof}[Proof of \cref{stabilizerscor}]
    Let $S$ denote the stabilizer of $U$ in $\SV_2(\bar{\QQ}\cap\RR)$, and consider an element $yhm \in U$, for $h\in H$ and $m \in M$. Then for every $s\in S$ we have $syhm \in U \subset yHM$, which implies $h^{-1} y^{-1}s yh \in HM$. Thus some conjugate $S'=gSg^{-1}$ (with $g\in G$) of $S$ is contained in $HM$, which implies that $\overline{S'}^Z \subset HM$ (due to the fact that $HM$ is Zariski closed). In particular the connected component of the identity of $\overline{S'}^Z$, denoted $[\overline{S'}^Z]^\circ$, is contained in $HM$. By \cref{finalClassification}, it is either contained in a conjugate of $H$ or it is conjugate to a small group. 
    
    If $[\overline{S'}^Z]^\circ$ is conjugate to a small group, then there exist $c \in G$ and a weakly small subgroup $\underline{P}$ of $\SV_2$, defined over $\bar{\QQ}\cap\RR$, such that $c[\overline{S'}^Z]^\circ c^{-1} \subset \underline{P}(\RR)$. Thus $cg[\overline{S}^Z]^\circ g^{-1}c^{-1} \subset \underline{P}(\RR)$. Note that $[\overline{S}^Z]^\circ$ is defined over $\bar{\Q} \cap \R$, by \cref{def_over_field}. So \cite[Lemma 18]{ShemTovSilberman:HH_AQUE_preprint} implies that $[\overline{S}^Z]^\circ \subset d \underline{P}(\RR) d^{-1}$ for some $d \in \SV_2(\bar{\Q}\cap \R)$. Thus $[\overline{S}^Z]^\circ$ is small. Since it is of finite index in $\overline{S}^Z$, note that $[\overline{S}^Z]^\circ \cap S$ has finite index in $S$, so $S$ is virtually small. 
    
    Otherwise, $[\overline{gSg^{-1}}^Z]^\circ\subset g'Hg'^{-1}$ for some $g'\in G$. Thus $[\overline{S}^Z]^\circ$ is contained in a conjugate of $H$ by some element of $G$. Since both $[\overline{S}^Z]^\circ$ and $H$ are defined over $\bar{\QQ}\cap\RR$ by \cref{def_over_field}, we have that $[\overline{S}^Z]^\circ$ is contained in a conjugate of $H$ by an element of $\SV_2(\bar{\QQ}\cap\RR)$ by {\cite[Lemma 18]{ShemTovSilberman:HH_AQUE_preprint}}, and the result follows. 
    
\end{proof}

\begin{proof}[Proof of \cref{maintech}]
    We need to show that any weak-$*$ limit $\tilde{\mu}$ of the sequence $\tilde{\mu}_n$ gives measure zero to $\Gamma yHM$, for any $y \in\SV_2(\bar{\QQ}\cap\RR)$. It suffices to prove that such a weak-$*$ limit gives zero measure to (the image in $\Gamma\bs G$ of) any bounded subset $U$ of $L = yHM$ contained in a fundamental domain for $\Gamma\bs G$. We prove the following claim, which implies the desired result. 
    
    \begin{claim*}
        For every bounded subset $U$ of $L$ contained in a fundamental domain for $\Gamma\bs G$, and for every $\alpha\in G$, there exists a compact neighborhood $B$ of the identity element in $G$ such that the set $U'=U\cap B\alpha$ satisfies the following property: 
        for any $\epsilon>0$ there exists $\delta>0$ such that $\tilde{\mu}_n(U'_\delta)\le\epsilon$ for every $n$.
    \end{claim*}
    
    The proof is by induction on the dimension of $U$. The base of the induction is $\dim U=-1$, that is $U$ is the empty set, which is vacuously true.  
    Let $U$ be any non-empty bounded subset of $L$ contained in a fundamental domain of $\Gamma\backslash G$. 
    
    We may assume without loss of generality that $U$ is compact by passing to the closure of $U$ (which does not change $\dim{U}$), and that $\bar{U}^Z$ is irreducible  by passing to the (compact) intersections $U_i = U \cap V_i$, where the $V_i$ denote the irreducible components of $\bar{U}^Z$. Indeed, note that either $\bar{U_i}^Z = V_i$ (hence $\bar{U_i}^Z$ is irreducible), or $\dim U_i := \dim \bar{U_i}^Z < \dim V_i \leq \dim \bar{U}^Z =: \dim U$, in which case we simply apply the induction hypothesis to $U_i$.
    
    By \cref{stabilizerscor}, either the stabilizer of $\bar{U}^Z$ in $\SV_2(\bar{\QQ}\cap\RR)$ is virtually small (see \cref{smallness-def}), or it is virtually contained in $gHg^{-1}$ for an element $g \in \SV_2(\bar{\QQ}\cap\RR)$. 
    Let $\alpha\in G$ be given as in the induction hypothesis. Fix compact neighborhoods $B$ and $\tilde{B}$ of the identity in $G$ such that $B$ is contained in the interior of $\tilde{B}$, and the intersection of $\tilde{B}\tilde{B}^{-1}$ with the stabilizer of $\bar{U}^Z$ is either small (rather than just virtually small) or contained in $gHg^{-1}$. Let $U'=U\cap B\alpha$. If $\dim{U'} < \dim{U}$ then by the induction hypothesis we have that $U'\cap B'\alpha=U\cap (B\cap B')\alpha$ satisfies the required property for some compact neighborhood $B'$ of the identity and we are done. Thus we may assume that $\dim{U'} = \dim{U}$, hence they have the same Zariski closure (since $\bar{U}^Z$ is irreducible).         
    
    If the intersection of the stabilizer of $\bar{U}^Z = \bar{U'}^Z$ with $\tilde{B} \tilde{B}^{-1}$ is small we apply \cite[Corollary 10]{ShemTovSilberman:AQUE4d}. If it is contained in a conjugate $gHg^{-1}$ we apply the bounds in \cref{lem:GlobalAmplifier}. Either way, combined with \cref{lemma:parallel_equals_stabilizer}, the conclusion is that for every $n$ there exists a positive definite Hecke operator $\tau_n = \tau_n(\varepsilon, g)$ satisfying $\norm{\tau_n}_\infty \ll_{\varepsilon, g} 1$ and with support contained on a set that depends only on $\{\varepsilon, g\}$ (and not on $n$) such that for every $\delta'>0$ small enough in terms of $U'$ (so that $U'_{\delta'}(U'_{\delta'})^{-1}\subset \tilde{B} \tilde{B}^{-1}$), we have
    \begin{equation}\label{ParBound}
        \frac{1}{\sL_n}\norm{\tau_n}_{L^1(\cP_n \cap U'_{\delta'}(U'_{\delta'})^{-1})} \le \epsilon,
    \end{equation} 
    where $\sL_n$ is the $\tau_n$-eigenvalue of $\psi_n$ and $\cP_n$ denotes the set of parallel elements of $\tau_n$ with respect to $U'$.  

    We now apply \cref{basic} (with $L = \bar{U'}^Z = \bar{U}^Z$ and $\tau = \tau_n$). Let $\cN_n$ denote the corresponding collection of compact subsets $N \subset U'_0$, which satisfy $\dim{N} < \dim{U'}$. By the induction hypothesis, 
    \begin{equation*}
        \lim_{\delta\to 0^+} \lim_{n \to \infty} \tilde{\mu}_{n}(N_\delta) = 0.
    \end{equation*}
    
    Since the support of $\tau_n$ is contained on a set that depends only on $\{\varepsilon, g\}$ (and not on $n$), \cref{basic} implies that by enlarging the collection $\cN_n$ we may choose it to be fixed as $n$ varies, and denote it by $\cN$. Then we can choose $\delta>0$ sufficiently small in terms of $\cN$ (hence depending only on $\{U', \varepsilon, g\}$) so that for every $n$ we have
    \begin{equation}\label{TranBound}
        \norm{\tau_n}_\infty \sum_{N\in\cN} \tilde{\mu}_{n}(N_\delta) \le \epsilon. 
    \end{equation}
    
    Combining \eqref{ParBound} and \eqref{TranBound}, we get by \cref{basic} that there exists $0 < \delta' < \delta$ such that for every $n$ we have
    \begin{equation*}
        \tilde{\mu}_n(U'_{\delta'})\le \epsilon+\frac{\epsilon}{\tilde{\mu}_n(U'_{\delta'})} \ll \sqrt{\epsilon}, 
    \end{equation*}
    and the result follows.
    
\end{proof}

%%%%%%%%%%%%%%%%%%%%%%%%%%%%%%%%%%%%%%%%%%%%%%%%%%%%%%%

\appendix

\crefalias{section}{appendix}

\section{Definition and properties of Pitale lifts}\label{section:Pitale_appendix}

\subsection{Maass forms of weight \texorpdfstring{$1/2$}{} on \texorpdfstring{$\H^2$}{}}

Let $S_{1/2}$ denote the space of Maass cusp forms of weight $1/2$ on the upper half-plane $\H^2 = \{z = x + \bi y \in \bH : x\in \R, y > 0\}$, which are smooth functions $f: \H^2 \to \C$ satisfying the following properties.
\begin{itemize}
    \item[(a)] 
    For all $\gamma \in \Gamma_0(4) = \left\{ \left(\begin{smallmatrix}
        a & b \\ 
        c & d
    \end{smallmatrix}\right)
    \in \SL_2(\Z) : 4 \mid c \right\}$ and $z \in \H^2$,
        \begin{equation*}
        f(\gamma z) = J(\gamma, z) f(z)
    \end{equation*}
    for the multiplier $J(\gamma, z) = \frac{\theta(\gamma z)}{\theta(z)}$ of $\theta(z) = y^{1/4} \sum_{n\in \Z} e(n^2 z)$, where $\gamma z = \frac{az+b}{cz+d}$.

    \item[(b)]
    We have $\Delta_{1/2} f + \lambda f = 0$ for some eigenvalue $\lambda \in \C$, where the Laplacian of weight $1/2$ is given by
    \begin{equation*}
        \Delta_{1/2} = y^2(\partial_x^2 + \partial_y^2) - \tfrac{1}{2} \bi y \partial_x.
    \end{equation*}

    \item[(c)] 
    We have $\norm{f}_2^2 := \int_{\Gamma_0(4) \backslash \H^2} |f|^2 \frac{dx dy}{y^2} < \infty$, and the function $f$ is cuspidal, i.e.\ has vanishing zero-th Fourier coefficient at each of the three cusps ($\infty$, $0$, and $1/2$) of $\Gamma_0(4) \backslash \H^2$.
    
\end{itemize}

Every $f \in S_{1/2}$ has a Fourier expansion (at $\infty$)
\begin{equation}\label{eq:half_integral_Fourier_expansion}
    f(z) = \sum_{0 \neq n \in \Z} b_f(n) W_{\frac{\sgn(n)}{4}, \frac{\bi r}{2}}(4 \pi |n|y) e(nx)
\end{equation}
for $\lambda = \frac{1}{4} + (\frac{r}{2})^2$ and $W_{\nu, \mu}$ the classical Whittaker function -- so in particular $W_{0,s}(4 \pi y) = 2 \sqrt{y} K_s(2\pi y)$. 

The Kohnen \cite{Koh80} plus space $S_{1/2}^+ \subset S_{1/2}$ consists of those $f \in S_{1/2}$ for which $b_f(n) = 0$ whenever $n \equiv 2, 3 \pmod{4}$. For each odd prime $p$, the Hecke operator $T_{p^2}: S_{1/2} \to S_{1/2}$ is defined by 
\begin{equation}\label{eq:half_integral_Hecke_op_def}
    (T_{p^2} f)(z) = \sum_{0 \neq n \in \Z} b^{(p)}_{f}(n) W_{\frac{\sgn(n)}{4}, \frac{\bi r}{2}}(4 \pi |n|y) e(nx)
\end{equation}
for 
\begin{equation*}
    b^{(p)}_{f}(n) := p b_f(np^2) +  p^{-1/2} \Big(\frac{n}{p}\Big) b_f(n) + p^{-1} b_f(n/p^2).
\end{equation*}
We call $f \in S^+_{1/2}$ a \emph{Hecke cusp form} if it is an eigenfunction of each $T_{p^2}$ for $p \neq 2$. Katok and Sarnak \cite[Proposition 1.4]{KS93} showed that the $T_{p^2}$ are self-adjoint, commute with each other and with $\Delta_{1/2}$, and map $S^+_{1/2}$ to itself. Thus $S_{1/2}^+$ has an orthonormal basis of Hecke cusp forms.

%%%%%%%%%%%%%%%%%%%%%%%%%%%%%%%%%%%%%%%%%%%%%%%%%%%%%%%

\subsection{Pitale lifts on \texorpdfstring{$\H^4$}{}}

For each $f \in S_{1/2}^+$ with $\Delta_{1/2}$-eigenvalue $\frac{1}{4} + (\frac{r}{2})^2$ and Fourier expansion \eqref{eq:half_integral_Fourier_expansion}, define $\Pit{f}: \H^4 \to \C$ by
\begin{equation*}
    \Pit{f}(z) = y^{3/2} \sum_{0 \not= \beta \in V^3(\ZZ)} A_{\Pit{f}}(\beta) K_{\bi r}(2 \pi |\beta| y) e\left(\mathrm{Re}(\beta z)\right),
\end{equation*}
where for $0 \neq \beta \in V^3(\Z)$ of the form $\beta = b_0 + b_1 \bi + b_2 \bj$ with $b_i \in \Z$ we write $2^u d = \gcd(b_0, b_1, b_2)$ for $u\geq 0$ and $d \geq 1$ odd, and set
\begin{equation*}
    A_{\Pit{f}}(\beta) := |\beta| \sum_{t = 0}^u (-1)^t 2^{t/2} \sum_{\substack{n \mid d \\ n > 0}} b_f \Big(\frac{-|\beta|^2}{(2^t n)^2}\Big) n^{-1/2}.
\end{equation*}
Pitale has shown \cite[Theorem 3.3]{Pit05} that $\Pit{f}$ is a Maass cusp form on $\SV_2(\Z) \backslash \H^4$ with $\Delta$-eigenvalue $(\frac{3}{2})^2+r^2$. Furthermore, the map $f \mapsto \Pit{f}$ is injective \cite[Theorem 4.6]{Pit05}.

If a non-zero $f \in S_{1/2}^+$ is a Hecke cusp form with $T_{p^2}f = \lambda_f(p) f$ for $p\neq 2$, then $\Pit{f}$ is a non-zero joint eigenfunction of the Hecke algebra $\cH$. Crucially for us, recalling the Hecke operators in \eqref{eq:T_p_matrices}, by \cite[Theorems 5.9 and 5.12]{Pit05} we obtain 
\begin{equation*}
    T_j(p) (\Pit{f}) = \Lambda_{j, f}(p) \cdot \Pit{f}
\end{equation*}
for $j\in \{1, 2\}$, where
\begin{equation*}
\Lambda_{1,f}(p)= p\big(p^{1/2} \lambda_f(p) + p + 1\big)\qquad \text{and} \qquad \Lambda_{2,f}(p)=(p+1)\big(p^{3/2} \lambda_f(p) + p - 1\big). 
\end{equation*}

\begin{definition}[Pitale lifts]\label{def:Pitale_lifts}
    Fix a basis $\{f_n\}_{n\geq 1}$ of the Kohnen plus space $S_{1/2}^+$ consisting of (non-zero) Hecke--Maass cusp forms of weight $1/2$, ordered by Laplacian eigenvalue. The sequence $\{\psi_n\}_{n\geq 1}$ of Pitale lifts is given by $\psi_n = \frac{\Pit{f_n}}{\norm{\Pit{f_n}}_2}$.
\end{definition}

Note that the lifted Hecke--Maass cusp forms $\psi_n$ on $\SV_2(\Z)\backslash \H^4$ are distinct, $L^2$-normalized, and also ordered by Laplacian eigenvalue.

%%%%%%%%%%%%%%%%%%%%%%%%%%%%%%%%%%%%%%%%%%%%%%%%%%%%%%%

\section{Computations of Hecke operators}\label{section:computation_appendix}

In this section we prove \cref{lem:DecompositionsUsingComputer}. In the notation of \cref{sec:HeckeTranslates}, the isomorphism $\psi_p$ gives identifications $G_p \isom \PGSp_4(\QQ_p)$ and $K_p \isom \PGSp_4(\ZZ_p)$, and by \cref{lemma:symplectic_Hecke_op}, $\tau_{m, l}(p)$ corresponds to the characteristic function of $K_p\diag(1,p^m,p^l,p^{l-m})K_p$. 

The basic idea for the proof of \cref{lem:DecompositionsUsingComputer} is simple: the algebra $\cH(G_p,K_p)$ is isomorphic (via the Satake isomorphism) to a certain algebra of rational functions in two variables. In particular, the images of the $K_p$-double cosets under this isomorphism form a linear basis for this algebra of rational functions. We compute these images explicitly, then decompose products of such basic rational functions as linear combinations in them. To compute the images of the double cosets we use Macdonald's formula for spherical functions. To decompose products of polynomials in terms of the basis elements we use a computer program \cite{code_repo}. 

Similar computations with certain operators in $\cH(G_p,K_p)$ can be found in work of Blomer and Pohl \cite[Sections 4 and 8]{BP16}, relying on explicit single-coset decompositions due to Kodama \cite{Kod67}. We used some of the more complicated relations found there, such as \cite[(4.9)]{BP16}, to test the correctness of our program. In \cite{BP16} the authors remark that
\begin{quote}
    ``It is tempting to perform all of these computations in the algebra of Weyl group invariant polynomials using the Satake isomorphism. Unfortunately, the computations are by no means easier, since it is very hard to compute explicitly the image of a given double coset.''
\end{quote}
We hope our program (and its extensions to other groups) will remove this computational difficulty in future applications.

%%%%%%%%%%%%%%%%%%%%%%%%%%%%%%%%%%%%%%%%%%%%%%%%%%%%%%%

\subsection{Spherical transform}

Let
\begin{equation*}
    T_{n,m,l}=\diag(p^n,p^m,p^{l-n},p^{l-m}),
\end{equation*}
which represent elements of the diagonal torus of $G_p$. Let $\cT$ be the lattice generated by the $T_{n,m,l}$, for $n, m, l \in \Z$. For each $f\in \cH(G_p,K_p)$ define the spherical transform
\begin{equation*}
    \widehat{f}(s) := (f * \omega_s)(1) = \int_{G_p}f(g)\omega_s(g^{-1})\; dg,
\end{equation*}
where $\omega_s$ denotes the spherical function with parameter $s\in\rm{Hom}(\cT,\CC^\times)$ (see \cite{Macdonald:SphericalFunctionsBook} for details about spherical functions). By \cite[Theorem (3.3.6)]{Macdonald:SphericalFunctionsBook}, the function $f\in \cH(G_p,K_p)$ is uniquely determined by its spherical transform $\widehat{f}$ and we have 
\begin{equation}\label{ConvProd}
    \widehat{f_1*f_2}=\widehat{f_1}\cdot \widehat{f_2}
\end{equation}
for every $f_1,f_2\in \cH(G_p,K_p)$. The spherical transform $\widehat{f}(s)$ can be viewed as a rational function in two variables: since we work modulo the center, a choice of $s\in \rm{{Hom}}(\cT,\CC^\times)$ is equivalent to a choice of $s(T_{0,1,0}) = Y$ and $s(T_{0,0,1}) = Z$. In those coordinates, since $s(T_{1, 1, 2}) = 1$ we have $s(T_{n, m, l}) = Y^{m-n} Z^{l-2n}$.

Choose a set of positive roots $\{\alpha_1,\alpha_2, \alpha_3,\alpha_4\}$ by defining their action on $T_{n,m,l}$ by
\begin{equation*}
    \alpha_1 = m-n, \qquad \alpha_2 = l-2n, \qquad \alpha_3 = l-n-m, \qquad \alpha_4 = l-2m.
\end{equation*}
The corresponding coroots $\check{\alpha}_i$ are cocharacters that map $p$ to $T_{-1,1,0},T_{-1,0,0},T_{-1,-1,0},T_{0,-1,0}$, respectively.

%%%%%%%%%%%%%%%%%%%%%%%%%%%%%%%%%%%%%%%%%%%%%%%%%%%%%%%

\subsection{Macdonald's formula for spherical functions}

We have the following formula for $\omega_s$, due to Macdonald \cite[Theorem (4.1.2)]{Macdonald:SphericalFunctionsBook} for simply connected groups, and Casselman \cite{Casselman1980} for general reductive groups (see also \cite{Casselman2012RemarksOM}). 
Call $T_{n,m,l}$ \emph{positive} if $\alpha_i(T_{n,m,l})\ge0$ for all $1\le i\le4$, which is equivalent to $n\le m\le l/2$. For positive $T_{n,m,l}$ we have
\begin{equation}\label{eq:MacDonald} 
    \omega_s(T_{-n,-m,-l})=Q(p^{-1})^{-1}p^{-\rho(T_{n,m,l})}\sum_{w\in W}ws(T_{n,m,l})c(ws), 
\end{equation}
where $\rho(T_{n,m,l})=\frac{1}{2}(3l-2m-4n)$, $Q$ is the Poincar{\'e} polynomial (see \cite[page 40]{Macdonald:SphericalFunctionsBook} for the definition), the Weyl group $W$ acts on $s$ by $ws(T_{n,m,l}) := s(w^{-1}T_{n,m,l})$ according to \eqref{eq:Weyl_group_gens}, and $c(s)$ is the Harish-Chandra function 
\begin{equation*}
    c(s)=\prod_{i=1}^4\frac{1-p^{-1}s(\check{\alpha}_i)^{-1}}{1-s(\check{\alpha}_i)^{-1}}.
\end{equation*}
The expression in \eqref{eq:MacDonald} is a Laurent polynomial in $Y$ and $Z$. Observe that 
\begin{equation}\label{eq:hecke_spherical_transf}
    \widehat{\mathbf{1}}_{K_pT_{n,m,l}K_p}(s)=\int_{K_pT_{n,m,l}K_p}\omega_s(g^{-1})\, dg=\vol(K_pT_{n,m,l}K_p)\omega_s(T_{-n,-m,-l}).
\end{equation}

%%%%%%%%%%%%%%%%%%%%%%%%%%%%%%%%%%%%%%%%%%%%%%%%%%%%%%%

\subsection{Volumes}

The volume $\vol(K_pT_{n,m,l}K_p)$ is computed using \cite[Chapter III]{Macdonald:SphericalFunctionsBook}. More precisely, set $C_0 := \{ T_{n,m,l} : n < m < l/2 \}$, and for $w \in W$ let
\begin{equation*}
    C_0(w) := \{ 1 \leq i \leq 4 : \alpha_i(wT) < 0 \text{ for all } T \in C_0\}.
\end{equation*}
For $t = T_{n,m,l}$ positive, denote $W_t := \{ w \in W : wt = t\}$ and
\begin{equation*}
    Q_t(x) := \sum_{w \in W_t} x^{|C_0(w)|}.
\end{equation*}
Then the Poincar{\'e} polynomial is
\begin{equation*}
    Q(x) = \sum_{w \in W} x^{|C_0(w)|} = x^4 + 2x^3 + 2x^2 + 2x + 1,
\end{equation*}
and \cite[Proposition (3.2.15)]{Macdonald:SphericalFunctionsBook} gives
\begin{equation*}
    \vol(K_p t K_p) = \frac{Q(p^{-1})}{Q_t(p^{-1})} p^{2 \rho(t)}.
\end{equation*}

%%%%%%%%%%%%%%%%%%%%%%%%%%%%%%%%%%%%%%%%%%%%%%%%%%%%%%%

\subsection{Decomposition into basic operators}

For any Hecke operator $\tau \in \cH(G_p, K_p)$ given as a polynomial on the basic operators $\tau_{m, l}(p)$, we can compute its spherical transform $\widehat{\tau}(s) =: \widehat{\tau}(Y, Z)$, which is a rational function of $Y$ and $Z$, by using \eqref{eq:hecke_spherical_transf} and \eqref{eq:MacDonald} for the indicator function of $K_p T_{0,m,l} K_p$, which corresponds to $\tau_{m,l}(p)$, and applying \eqref{ConvProd}. 

With explicit rational functions at hand for $\widehat{\tau}(Y, Z)$ and each $\widehat{\tau}_{m,l}(Y, Z)$, we can search\footnote{Such a relation is guaranteed to exist, since it is merely an expansion into a linear basis.} for a finite relation in $\C[Y^{\pm 1}, Z^{\pm 1}]$ of the form
\begin{equation*}
    \widehat{\tau}(Y, Z) = \sum_{l\geq 2m \geq 0} b_{m,l} \cdot \widehat{\tau}_{m,l}(Y, Z)
\end{equation*}
with $b_{m,l} \in \C$. Once such a relation is found, injectivity of the spherical transform gives the desired decomposition into basic operators,
\begin{equation*}
    \tau = \sum_{l\geq 2m \geq 0} b_{m,l} \cdot \tau_{m,l}(p).
\end{equation*}

%%%%%%%%%%%%%%%%%%%%%%%%%%%%%%%%%%%%%%%%%%%%%%%%%%%%%%%

\subsection{Proof of \texorpdfstring{\cref{lem:DecompositionsUsingComputer}}{}}

We implemented the algorithm described above in SageMath. The code is available at \cite{code_repo}, and using it we obtain precisely the relations given in \cref{lem:DecompositionsUsingComputer}. For the reader's convenience, let us give the spherical transforms of all the relevant operators. We compute
{\tiny
\begin{align*}
    \autoalign{\widehat{T^2_2}(Y, Z) = p^{4} Y^{2} Z^{4} + 2 p^{4} Y^{2} Z^{2} + 2 p^{4} Y Z^{2} + p^{4} Y^{2} + 2 p^{4} Z^{2} + 2 p^{4} Y - 2 p^{2} Y Z^{2} + 5 p^{4} + 2 p^{4} Y^{-1} - 2 p^{2} Y + 2 p^{4} Z^{-2} + p^{4} Y^{-2} - 2 p^{2} + 2 p^{4} Y^{-1} Z^{-2} - 2 p^{2} Y^{-1} + 2 p^{4} Y^{-2} Z^{-2} + 1 - 2 p^{2} Y^{-1} Z^{-2} + p^{4} Y^{-2} Z^{-4},}
\end{align*}

\begin{align*}
    \autoalign{\widehat{\sigma}(Y, Z) = p^{4} Y^{2} Z^{4} + p^{4} Y^{2} Z^{2} - p^{3} Y^{2} Z^{2} + p^{4} Y^{2} + p^{4} Z^{2} - 2 p^{3} Y Z^{2} - p^{3} Z^{2} - 2 p^{2} Y Z^{2} + p^{4} - 2 p^{3} Y - 4 p^{3} - 2 p^{2} Y + p^{4} Z^{-2} + p^{4} Y^{-2} - 2 p^{3} Y^{-1} - 2 p^{2} - p^{3} Z^{-2} - 2 p^{2} Y^{-1} + p^{4} Y^{-2} Z^{-2} - 2 p^{3} Y^{-1} Z^{-2} + 1 - p^{3} Y^{-2} Z^{-2} - 2 p^{2} Y^{-1} Z^{-2} + p^{4} Y^{-2} Z^{-4},} 
\end{align*}

\begin{align*}
    \autoalign{\widehat{\sigma^2}(Y, Z) = p^{8} Y^{4} Z^{8} + 2 p^{8} Y^{4} Z^{6} - 2 p^{7} Y^{4} Z^{6} + 3 p^{8} Y^{4} Z^{4} + 2 p^{8} Y^{2} Z^{6} - 4 p^{7} Y^{3} Z^{6} - 2 p^{7} Y^{4} Z^{4} - 2 p^{7} Y^{2} Z^{6} - 4 p^{6} Y^{3} Z^{6} + 2 p^{8} Y^{4} Z^{2} + 4 p^{8} Y^{2} Z^{4} - 8 p^{7} Y^{3} Z^{4} + p^{6} Y^{4} Z^{4} - 2 p^{7} Y^{4} Z^{2} - 12 p^{7} Y^{2} Z^{4} - 4 p^{6} Y^{3} Z^{4} + p^{8} Y^{4} + 6 p^{8} Y^{2} Z^{2} - 8 p^{7} Y^{3} Z^{2} + 3 p^{8} Z^{4} - 8 p^{7} Y Z^{4} + 2 p^{6} Y^{2} Z^{4} + 4 p^{5} Y^{3} Z^{4} - 14 p^{7} Y^{2} Z^{2} - 4 p^{6} Y^{3} Z^{2} - 2 p^{7} Z^{4} - 4 p^{6} Y Z^{4} + 8 p^{5} Y^{2} Z^{4} + 4 p^{8} Y^{2} - 4 p^{7} Y^{3} + 6 p^{8} Z^{2} - 16 p^{7} Y Z^{2} + 12 p^{6} Y^{2} Z^{2} + 4 p^{5} Y^{3} Z^{2} + p^{6} Z^{4} + 4 p^{5} Y Z^{4} + 6 p^{4} Y^{2} Z^{4} - 12 p^{7} Y^{2} - 4 p^{6} Y^{3} - 14 p^{7} Z^{2} + 8 p^{6} Y Z^{2} + 20 p^{5} Y^{2} Z^{2} + 2 p^{8} Y^{2} Z^{-2} + 9 p^{8} - 16 p^{7} Y + 2 p^{6} Y^{2} + 2 p^{8} Y^{-2} Z^{2} - 8 p^{7} Y^{-1} Z^{2} + 12 p^{6} Z^{2} + 32 p^{5} Y Z^{2} + 10 p^{4} Y^{2} Z^{2} - 2 p^{7} Y^{2} Z^{-2} - 16 p^{7} + 8 p^{6} Y + 8 p^{5} Y^{2} - 2 p^{7} Y^{-2} Z^{2} - 4 p^{6} Y^{-1} Z^{2} + 20 p^{5} Z^{2} + 8 p^{4} Y Z^{2} - 2 p^{3} Y^{2} Z^{2} + 6 p^{8} Z^{-2} - 8 p^{7} Y Z^{-2} + 4 p^{8} Y^{-2} - 16 p^{7} Y^{-1} + 32 p^{6} + 32 p^{5} Y + 6 p^{4} Y^{2} + 4 p^{5} Y^{-1} Z^{2} + 10 p^{4} Z^{2} - 4 p^{3} Y Z^{2} - 14 p^{7} Z^{-2} - 4 p^{6} Y Z^{-2} - 12 p^{7} Y^{-2} + 8 p^{6} Y^{-1} + 48 p^{5} + 8 p^{4} Y - 2 p^{3} Z^{2} - 4 p^{2} Y Z^{2} + 3 p^{8} Z^{-4} + 6 p^{8} Y^{-2} Z^{-2} - 16 p^{7} Y^{-1} Z^{-2} + 12 p^{6} Z^{-2} + 4 p^{5} Y Z^{-2} + p^{8} Y^{-4} - 4 p^{7} Y^{-3} + 2 p^{6} Y^{-2} + 32 p^{5} Y^{-1} + 22 p^{4} - 4 p^{3} Y - 2 p^{7} Z^{-4} - 14 p^{7} Y^{-2} Z^{-2} + 8 p^{6} Y^{-1} Z^{-2} + 20 p^{5} Z^{-2} - 4 p^{6} Y^{-3} + 8 p^{5} Y^{-2} + 8 p^{4} Y^{-1} - 8 p^{3} - 4 p^{2} Y + 4 p^{8} Y^{-2} Z^{-4} - 8 p^{7} Y^{-1} Z^{-4} + p^{6} Z^{-4} + 2 p^{8} Y^{-4} Z^{-2} - 8 p^{7} Y^{-3} Z^{-2} + 12 p^{6} Y^{-2} Z^{-2} + 32 p^{5} Y^{-1} Z^{-2} + 10 p^{4} Z^{-2} + 6 p^{4} Y^{-2} - 4 p^{3} Y^{-1} - 4 p^{2} - 12 p^{7} Y^{-2} Z^{-4} - 4 p^{6} Y^{-1} Z^{-4} - 2 p^{7} Y^{-4} Z^{-2} - 4 p^{6} Y^{-3} Z^{-2} + 20 p^{5} Y^{-2} Z^{-2} + 8 p^{4} Y^{-1} Z^{-2} - 2 p^{3} Z^{-2} - 4 p^{2} Y^{-1} + 2 p^{8} Y^{-2} Z^{-6} + 3 p^{8} Y^{-4} Z^{-4} - 8 p^{7} Y^{-3} Z^{-4} + 2 p^{6} Y^{-2} Z^{-4} + 4 p^{5} Y^{-1} Z^{-4} + 4 p^{5} Y^{-3} Z^{-2} + 10 p^{4} Y^{-2} Z^{-2} - 4 p^{3} Y^{-1} Z^{-2} + 1 - 2 p^{7} Y^{-2} Z^{-6} - 2 p^{7} Y^{-4} Z^{-4} - 4 p^{6} Y^{-3} Z^{-4} + 8 p^{5} Y^{-2} Z^{-4} - 2 p^{3} Y^{-2} Z^{-2} - 4 p^{2} Y^{-1} Z^{-2} + 2 p^{8} Y^{-4} Z^{-6} - 4 p^{7} Y^{-3} Z^{-6} + p^{6} Y^{-4} Z^{-4} + 4 p^{5} Y^{-3} Z^{-4} + 6 p^{4} Y^{-2} Z^{-4} - 2 p^{7} Y^{-4} Z^{-6} - 4 p^{6} Y^{-3} Z^{-6} + p^{8} Y^{-4} Z^{-8}
.}
\end{align*}
}
We also compute the spherical transforms of a few basic Hecke operators, obtaining
{\tiny
\begin{align*}
     \widehat{\tau}_{0,0}=1, \qquad \qquad \qquad \widehat{\tau}_{1,2} = p^{2} Y Z^{2} + p^{2} Y + p^{2} + p^{2} Y^{-1} - 1 + p^{2} Y^{-1} Z^{-2},
\end{align*}

\begin{align*}
    \widehat{\tau}_{0,2} = p^{3} Y^{2} Z^{2} + p^{3} Y Z^{2} + p^{3} Z^{2} - p^{2} Y Z^{2} + p^{3} Y + 2 p^{3} - p^{2} Y + p^{3} Y^{-1} \\
    - 2 p^{2} + p^{3} Z^{-2} - p^{2} Y^{-1} + p^{3} Y^{-1} Z^{-2} + p^{3} Y^{-2} Z^{-2} - p^{2} Y^{-1} Z^{-2},
\end{align*}

\begin{align*}
    \autoalign{\widehat{\tau}_{0,4} = p^{6} Y^{4} Z^{4} + p^{6} Y^{3} Z^{4} + p^{6} Y^{2} Z^{4} - p^{5} Y^{3} Z^{4} + p^{6} Y^{3} Z^{2} + p^{6} Y Z^{4} - p^{5} Y^{2} Z^{4} + 2 p^{6} Y^{2} Z^{2} - p^{5} Y^{3} Z^{2} + p^{6} Z^{4} - p^{5} Y Z^{4} + 2 p^{6} Y Z^{2} - 3 p^{5} Y^{2} Z^{2} + p^{6} Y^{2} + 2 p^{6} Z^{2} - 4 p^{5} Y Z^{2} + p^{4} Y^{2} Z^{2} + 2 p^{6} Y - p^{5} Y^{2} + p^{6} Y^{-1} Z^{2} - 3 p^{5} Z^{2} + 2 p^{4} Y Z^{2} + 3 p^{6} - 4 p^{5} Y - p^{5} Y^{-1} Z^{2} + p^{4} Z^{2} + p^{6} Y Z^{-2} + 2 p^{6} Y^{-1} - 5 p^{5} + 2 p^{4} Y + 2 p^{6} Z^{-2} - p^{5} Y Z^{-2} + p^{6} Y^{-2} - 4 p^{5} Y^{-1} + 3 p^{4} + 2 p^{6} Y^{-1} Z^{-2} - 3 p^{5} Z^{-2} - p^{5} Y^{-2} + 2 p^{4} Y^{-1} - p^{3} + p^{6} Z^{-4} + 2 p^{6} Y^{-2} Z^{-2} - 4 p^{5} Y^{-1} Z^{-2} + p^{4} Z^{-2} + p^{6} Y^{-1} Z^{-4} + p^{6} Y^{-3} Z^{-2} - 3 p^{5} Y^{-2} Z^{-2} + 2 p^{4} Y^{-1} Z^{-2} + p^{6} Y^{-2} Z^{-4} - p^{5} Y^{-1} Z^{-4} - p^{5} Y^{-3} Z^{-2} + p^{4} Y^{-2} Z^{-2} + p^{6} Y^{-3} Z^{-4} - p^{5} Y^{-2} Z^{-4} + p^{6} Y^{-4} Z^{-4} - p^{5} Y^{-3} Z^{-4},}
\end{align*}

\begin{align*}
    \autoalign{\widehat{\tau}_{1,4} = p^{5} Y^{3} Z^{4} + p^{5} Y^{2} Z^{4} + p^{5} Y^{3} Z^{2} + p^{5} Y Z^{4} - p^{4} Y^{2} Z^{4} + 2 p^{5} Y^{2} Z^{2} + 3 p^{5} Y Z^{2} - 2 p^{4} Y^{2} Z^{2} + p^{5} Y^{2} + 2 p^{5} Z^{2} - 3 p^{4} Y Z^{2} + 3 p^{5} Y - p^{4} Y^{2} + p^{5} Y^{-1} Z^{2} - 2 p^{4} Z^{2} + 3 p^{5} - 3 p^{4} Y + p^{5} Y Z^{-2} + 3 p^{5} Y^{-1} - 5 p^{4} + 2 p^{5} Z^{-2} + p^{5} Y^{-2} - 3 p^{4} Y^{-1} + p^{3} + 3 p^{5} Y^{-1} Z^{-2} - 2 p^{4} Z^{-2} - p^{4} Y^{-2} + p^{2} + 2 p^{5} Y^{-2} Z^{-2} - 3 p^{4} Y^{-1} Z^{-2} + p^{5} Y^{-1} Z^{-4} + p^{5} Y^{-3} Z^{-2} - 2 p^{4} Y^{-2} Z^{-2} + p^{5} Y^{-2} Z^{-4} + p^{5} Y^{-3} Z^{-4} - p^{4} Y^{-2} Z^{-4},}
\end{align*}

\begin{align*}
    \autoalign{\widehat{\tau}_{2,4} = p^{4} Y^{2} Z^{4} + p^{4} Y^{2} Z^{2} + p^{4} Y Z^{2} - p^{3} Y^{2} Z^{2} + p^{4} Y^{2} + p^{4} Z^{2} - p^{3} Y Z^{2} + p^{4} Y - p^{3} Z^{2} + 2 p^{4} - p^{3} Y + p^{4} Y^{-1} - 2 p^{3} + p^{4} Z^{-2} + p^{4} Y^{-2} - p^{3} Y^{-1} + p^{4} Y^{-1} Z^{-2} - p^{3} Z^{-2} + p^{4} Y^{-2} Z^{-2} - p^{3} Y^{-1} Z^{-2} - p^{3} Y^{-2} Z^{-2} + p^{4} Y^{-2} Z^{-4},}
\end{align*}

\begin{align*}
    \autoalign{\widehat{\tau}_{2,6} = p^{7} Y^{4} Z^{6} + p^{7} Y^{3} Z^{6} + p^{7} Y^{4} Z^{4} + p^{7} Y^{2} Z^{6} - p^{6} Y^{3} Z^{6} + 2 p^{7} Y^{3} Z^{4} - p^{6} Y^{4} Z^{4} + p^{7} Y^{4} Z^{2} + 3 p^{7} Y^{2} Z^{4} - 3 p^{6} Y^{3} Z^{4} + 2 p^{7} Y^{3} Z^{2} + 2 p^{7} Y Z^{4} - 4 p^{6} Y^{2} Z^{4} + p^{5} Y^{3} Z^{4} + 4 p^{7} Y^{2} Z^{2} - 3 p^{6} Y^{3} Z^{2} + p^{7} Z^{4} - 3 p^{6} Y Z^{4} + p^{5} Y^{2} Z^{4} + p^{7} Y^{3} + 4 p^{7} Y Z^{2} - 6 p^{6} Y^{2} Z^{2} + p^{5} Y^{3} Z^{2} - p^{6} Z^{4} + p^{5} Y Z^{4} + 3 p^{7} Y^{2} - p^{6} Y^{3} + 4 p^{7} Z^{2} - 8 p^{6} Y Z^{2} + 3 p^{5} Y^{2} Z^{2} + 4 p^{7} Y - 4 p^{6} Y^{2} + 2 p^{7} Y^{-1} Z^{2} - 6 p^{6} Z^{2} + 5 p^{5} Y Z^{2} - p^{4} Y^{2} Z^{2} + p^{7} Y^{2} Z^{-2} + 5 p^{7} - 8 p^{6} Y + p^{5} Y^{2} + p^{7} Y^{-2} Z^{2} - 3 p^{6} Y^{-1} Z^{2} + 3 p^{5} Z^{2} - p^{4} Y Z^{2} + 2 p^{7} Y Z^{-2} + 4 p^{7} Y^{-1} - 10 p^{6} + 5 p^{5} Y + p^{5} Y^{-1} Z^{2} - p^{4} Z^{2} + 4 p^{7} Z^{-2} - 3 p^{6} Y Z^{-2} + 3 p^{7} Y^{-2} - 8 p^{6} Y^{-1} + 6 p^{5} - p^{4} Y + 4 p^{7} Y^{-1} Z^{-2} - 6 p^{6} Z^{-2} + p^{5} Y Z^{-2} + p^{7} Y^{-3} - 4 p^{6} Y^{-2} + 5 p^{5} Y^{-1} - 2 p^{4} + p^{7} Z^{-4} + 4 p^{7} Y^{-2} Z^{-2} - 8 p^{6} Y^{-1} Z^{-2} + 3 p^{5} Z^{-2} - p^{6} Y^{-3} + p^{5} Y^{-2} - p^{4} Y^{-1} + p^{3} + 2 p^{7} Y^{-1} Z^{-4} - p^{6} Z^{-4} + 2 p^{7} Y^{-3} Z^{-2} - 6 p^{6} Y^{-2} Z^{-2} + 5 p^{5} Y^{-1} Z^{-2} - p^{4} Z^{-2} + 3 p^{7} Y^{-2} Z^{-4} - 3 p^{6} Y^{-1} Z^{-4} + p^{7} Y^{-4} Z^{-2} - 3 p^{6} Y^{-3} Z^{-2} + 3 p^{5} Y^{-2} Z^{-2} - p^{4} Y^{-1} Z^{-2} + 2 p^{7} Y^{-3} Z^{-4} - 4 p^{6} Y^{-2} Z^{-4} + p^{5} Y^{-1} Z^{-4} + p^{5} Y^{-3} Z^{-2} - p^{4} Y^{-2} Z^{-2} + p^{7} Y^{-2} Z^{-6} + p^{7} Y^{-4} Z^{-4} - 3 p^{6} Y^{-3} Z^{-4} + p^{5} Y^{-2} Z^{-4} + p^{7} Y^{-3} Z^{-6} - p^{6} Y^{-4} Z^{-4} + p^{5} Y^{-3} Z^{-4} + p^{7} Y^{-4} Z^{-6} - p^{6} Y^{-3} Z^{-6},}
\end{align*}

\begin{align*}
    \autoalign{\widehat{\tau}_{3,6} = p^{6} Y^{3} Z^{6} + p^{6} Y^{3} Z^{4} + p^{6} Y^{2} Z^{4} - p^{5} Y^{3} Z^{4} + p^{6} Y^{3} Z^{2} + p^{6} Y Z^{4} - p^{5} Y^{2} Z^{4} + p^{6} Y^{2} Z^{2} - p^{5} Y^{3} Z^{2} - p^{5} Y Z^{4} + p^{6} Y^{3} + 2 p^{6} Y Z^{2} - 2 p^{5} Y^{2} Z^{2} + p^{6} Y^{2} + p^{6} Z^{2} - 3 p^{5} Y Z^{2} + p^{4} Y^{2} Z^{2} + 2 p^{6} Y - p^{5} Y^{2} + p^{6} Y^{-1} Z^{2} - 2 p^{5} Z^{2} + p^{4} Y Z^{2} + 2 p^{6} - 3 p^{5} Y - p^{5} Y^{-1} Z^{2} + p^{4} Z^{2} + p^{6} Y Z^{-2} + 2 p^{6} Y^{-1} - 3 p^{5} + p^{4} Y + p^{6} Z^{-2} - p^{5} Y Z^{-2} + p^{6} Y^{-2} - 3 p^{5} Y^{-1} + 2 p^{4} + 2 p^{6} Y^{-1} Z^{-2} - 2 p^{5} Z^{-2} + p^{6} Y^{-3} - p^{5} Y^{-2} + p^{4} Y^{-1} - p^{3} + p^{6} Y^{-2} Z^{-2} - 3 p^{5} Y^{-1} Z^{-2} + p^{4} Z^{-2} + p^{6} Y^{-1} Z^{-4} + p^{6} Y^{-3} Z^{-2} - 2 p^{5} Y^{-2} Z^{-2} + p^{4} Y^{-1} Z^{-2} + p^{6} Y^{-2} Z^{-4} - p^{5} Y^{-1} Z^{-4} - p^{5} Y^{-3} Z^{-2} + p^{4} Y^{-2} Z^{-2} + p^{6} Y^{-3} Z^{-4} - p^{5} Y^{-2} Z^{-4} - p^{5} Y^{-3} Z^{-4} + p^{6} Y^{-3} Z^{-6},}
\end{align*}

\begin{align*}
    \autoalign{\widehat{\tau}_{4,8} = p^{8} Y^{4} Z^{8} + p^{8} Y^{4} Z^{6} + p^{8} Y^{3} Z^{6} - p^{7} Y^{4} Z^{6} + p^{8} Y^{4} Z^{4} + p^{8} Y^{2} Z^{6} - p^{7} Y^{3} Z^{6} + p^{8} Y^{3} Z^{4} - p^{7} Y^{4} Z^{4} - p^{7} Y^{2} Z^{6} + p^{8} Y^{4} Z^{2} + 2 p^{8} Y^{2} Z^{4} - 2 p^{7} Y^{3} Z^{4} + p^{8} Y^{3} Z^{2} - p^{7} Y^{4} Z^{2} + p^{8} Y Z^{4} - 3 p^{7} Y^{2} Z^{4} + p^{6} Y^{3} Z^{4} + p^{8} Y^{4} + 2 p^{8} Y^{2} Z^{2} - 2 p^{7} Y^{3} Z^{2} + p^{8} Z^{4} - 2 p^{7} Y Z^{4} + p^{6} Y^{2} Z^{4} + p^{8} Y^{3} + 2 p^{8} Y Z^{2} - 4 p^{7} Y^{2} Z^{2} + p^{6} Y^{3} Z^{2} - p^{7} Z^{4} + p^{6} Y Z^{4} + 2 p^{8} Y^{2} - p^{7} Y^{3} + 2 p^{8} Z^{2} - 4 p^{7} Y Z^{2} + 2 p^{6} Y^{2} Z^{2} + 2 p^{8} Y - 3 p^{7} Y^{2} + p^{8} Y^{-1} Z^{2} - 4 p^{7} Z^{2} + 3 p^{6} Y Z^{2} + p^{8} Y^{2} Z^{-2} + 3 p^{8} - 4 p^{7} Y + p^{6} Y^{2} + p^{8} Y^{-2} Z^{2} - 2 p^{7} Y^{-1} Z^{2} + 2 p^{6} Z^{2} - p^{5} Y Z^{2} + p^{8} Y Z^{-2} - p^{7} Y^{2} Z^{-2} + 2 p^{8} Y^{-1} - 5 p^{7} + 3 p^{6} Y - p^{7} Y^{-2} Z^{2} + p^{6} Y^{-1} Z^{2} + 2 p^{8} Z^{-2} - 2 p^{7} Y Z^{-2} + 2 p^{8} Y^{-2} - 4 p^{7} Y^{-1} + 3 p^{6} - p^{5} Y + 2 p^{8} Y^{-1} Z^{-2} - 4 p^{7} Z^{-2} + p^{6} Y Z^{-2} + p^{8} Y^{-3} - 3 p^{7} Y^{-2} + 3 p^{6} Y^{-1} - p^{5} + p^{8} Z^{-4} + 2 p^{8} Y^{-2} Z^{-2} - 4 p^{7} Y^{-1} Z^{-2} + 2 p^{6} Z^{-2} + p^{8} Y^{-4} - p^{7} Y^{-3} + p^{6} Y^{-2} - p^{5} Y^{-1} + p^{8} Y^{-1} Z^{-4} - p^{7} Z^{-4} + p^{8} Y^{-3} Z^{-2} - 4 p^{7} Y^{-2} Z^{-2} + 3 p^{6} Y^{-1} Z^{-2} + 2 p^{8} Y^{-2} Z^{-4} - 2 p^{7} Y^{-1} Z^{-4} + p^{8} Y^{-4} Z^{-2} - 2 p^{7} Y^{-3} Z^{-2} + 2 p^{6} Y^{-2} Z^{-2} - p^{5} Y^{-1} Z^{-2} + p^{8} Y^{-3} Z^{-4} - 3 p^{7} Y^{-2} Z^{-4} + p^{6} Y^{-1} Z^{-4} - p^{7} Y^{-4} Z^{-2} + p^{6} Y^{-3} Z^{-2} + p^{8} Y^{-2} Z^{-6} + p^{8} Y^{-4} Z^{-4} - 2 p^{7} Y^{-3} Z^{-4} + p^{6} Y^{-2} Z^{-4} + p^{8} Y^{-3} Z^{-6} - p^{7} Y^{-2} Z^{-6} - p^{7} Y^{-4} Z^{-4} + p^{6} Y^{-3} Z^{-4} + p^{8} Y^{-4} Z^{-6} - p^{7} Y^{-3} Z^{-6} - p^{7} Y^{-4} Z^{-6} + p^{8} Y^{-4} Z^{-8},}
\end{align*}
}
Checking the desired relations given in \cref{lem:DecompositionsUsingComputer} then becomes a tedious but straightforward task, which is also handled by our program.

%%%%%%%%%%%%%%%%%%%%%%%%%%%%%%%%%%%%%%%%%%%%%%%%%%%%%%%

\bibliographystyle{abbrv}
\bibliography{references}

\end{document}